\documentclass[11pt]{amsart}
\usepackage{mathrsfs}
\usepackage{amsfonts}
\usepackage{amsthm}
\usepackage{amssymb}
\usepackage{latexsym}
\usepackage{color}
\usepackage{xcolor}
\usepackage{amsmath, amsfonts, amsthm,  color, enumerate, cite, latexsym, mathrsfs}
\usepackage[T1]{fontenc}

\newtheorem{corollary}{Corollary}

\newtheorem{proposition}{Proposition}
\newcommand{\dbar}{\overline{\partial}}

\newtheorem{theorem}{Theorem}

\newtheorem*{notation*}{Notation}
{\theoremstyle{definition}

\DeclareMathOperator{\im}{Im}
\DeclareMathOperator{\dom}{dom}

\DeclareMathOperator{\supp}{supp}

\numberwithin{equation}{section}
{\theoremstyle{definition}

\newcommand{\df}{\text{Diederich-Forn{\ae}ss }}

\newcommand{\C}{\mathbb{C}}
\newcommand{\R}{\mathbb{R}}
\def\F{{\mathcal{F}}}
\renewcommand{\o}{\omega}
\renewcommand{\O}{\Omega}

\newcommand{\abs}[1]{\left|  #1 \right|}
\newcommand{\norm}[1]{\left\Vert#1\right\Vert}
\newcommand{\isum}{\sideset{}{'}\sum}

\newcommand{\dbarstar}{\bar{\partial}^{\star}}
\newcommand{\dbarb}{\bar{\partial}_b}
\newcommand{\dbarbstar}{\bar{\partial}_b^{\star}}

\newcommand{\dbarbstarweighted}{\bar{\partial}_{b,\varphi}^{\star}}
\newcommand{\dbarbstarwh}{\bar{\partial}_{b,-h_\eta}^{\star}}

\newcommand{\normw}[1]{\Vert#1\Vert_\varphi}
\newcommand{\normwh}[1]{\Vert#1\Vert_{-h_\eta}}
\newcommand{\weight}{e^{-\varphi}}

\newcommand{\bigo}[1]{\mathcal{O}(#1)}

\newcommand{\chip}{\chi^{+,\eta}}
\newcommand{\chim}{\chi^{-,\eta}}
\newcommand{\chiz}{\chi^{0,\eta}}

\newcommand{\chipt}{\tilde{\chi}^{+,\eta}}

\newcommand{\pplus}{\mathcal{P}^{+,\eta}}
\newcommand{\pminus}{\mathcal{P}^{-,\eta}}
\newcommand{\pzero}{\mathcal{P}^{0,\eta}}

\begin{document}

\title[regularity of the complex Green operator]{Diederich--Forn\ae ss index and global regularity of the complex Green operator: domains with comparable Levi eigenvalues}
\author{Tanuj Gupta}
\author{Emil J. Straube}
\subjclass[2010]{}
\thanks{Work supported in part by NSF grant DMS--2247175. The bulk of this material comes from the first author's Ph.~D. dissertation at Texas A\&M University (\cite{Gupta25}), supervised by the second author.}

\address{Department of Mathematics, Texas A\&M University, College Station, Texas, USA}
\email{tanujgupta17@tamu.edu}
\email{e-straube@tamu.edu}
\date{December 2025}

\begin{abstract}
Let $\Omega\subset\C^n$, with $n \geq 3$, be a smooth bounded pseudoconvex domain satisfying the symmetric eigenvalue comparability condition $D(q_0)$ for some $1\le q_0\le n-2$. We show that if the \df index of $\O$ is one, then the complex Green operator $G_q$, associated with $\O$, is globally regular for $q$ in the range $\min\{q_0,\, n - 1 - q_0\} \leq q \leq \max\{q_0,\, n - 1 - q_0\}$.
\end{abstract}

\maketitle
\section{Introduction}\label{intro} 
The study of regularity estimates for the $\overline{\partial}$--Neumann and complex Green operators in terms of (weak) plurisubharmonicity properties of a defining function goes back to \cite{Boas-Straube-1991, BoasStraube91b}, where Boas and the second author established Sobolev estimates for these operators when the domain admits a defining function that is plurisubharmonic at points of the boundary. Kohn (\cite{Kohn99}) then initiated a quantitative study of estimates in terms of the Diederich--Forn\ae ss index. The latter can be viewed as a measure of how close a domain is to having a defining function that is plurisubharmonic at the boundary (albeit with caution: index one does not imply the existence of such a defining function). These ideas were developed further by Harrington (\cite{Harrington11, Harrington19, Harrington22}), Pinton--Zampieri (\cite{PinZamp14}), and Liu (\cite{Liu22}). Most recently, Liu and the second author (\cite{Liu-Straube-2025}) showed that if $1 \leq q_0 \leq n-1$ and the $q_0$-sums of the eigenvalues of the Levi form of a smooth bounded pseudoconvex domain $\Omega$ are comparable, then Diederich--Forn\ae ss index one implies global regularity of the $\bar{\partial}$-Neumann operators $N_q$ and the Bergman projections $P_{q-1}$ for $q_0 \leq q \leq n$. All these results point towards the conjecture that index one should imply regularity. We refer the reader to Section 1 of \cite{Straube25b} for further information on these developments.

\smallskip

Here we prove an analogue of \cite{Liu-Straube-2025} for the complex Green operators $G_q$ and the Szeg{\"o} projections $S_{q-1}$. Our main result is the following.
\begin{theorem}\label{main-result}
Let $\O\subset\C^n$, with $n\geq 3$, be a smooth bounded pseudoconvex domain. Let $q_0$ be an integer such that $1\le q_0 \le n-2$ and assume that $\Omega$ satisfies condition $D(q_0)$. If $DF(\Omega) =1$, then for each $s\ge 0$, there exists a constant $C_s$ such that all for $u\in L^2_{(0,q)}(b\O)$ and $\min\{ q_0,n-1-q_0 \} \le q \le \max\{ q_0,n-1-q_0 \}$:
\begin{enumerate}
    \item $\norm{u}_s \le C_s ( \norm{\dbarb u}_s + \norm{\dbarbstar u}_s )$,
    \item $\norm{S_q u}_s \le C_s \norm{u}_s$\;, and
    \item $\norm{G_q u}_s \le C_s \norm{u}_s$\;.
\end{enumerate}
(1) and (2) also hold for $q=\min\{ q_0,n-1-q_0 \}-1$ and $q=\max\{ q_0,n-1-q_0 \}+1$ if in (1), $u\perp\ker(\overline{\partial}_{b})$ when $q=\min\{ q_0,n-1-q_0 \}-1$, and $u\perp\ker(\overline{\partial}_{b}^{*})$ when $q=\max\{ q_0,n-1-q_0 \}+1$, and if in (2), $S_{n-1}$ is the projection onto $\im(\overline{\partial}_{b})=\ker(\overline{\partial}_{b}^{*})^{\perp}$.
\end{theorem}
\noindent $DF(\Omega)$ refers to the Diederich--Forn\ae ss index of $\Omega$, and $D(q_{0})$ is a comparability condition bearing on sums of $q_{0}$ eigenvalues of the Levi form introduced by Derridj (\cite{Derridj-1991}) (see also Koenig (\cite{Koenig-2002}): precise definitions for both, as well as for the related condition $\widetilde{D}(q)$, are in Section \ref{def}. Observe that $\min\{ q_0,n-1-q_0 \}$ or $\max\{ q_0,n-1-q_0 \}$ equals $q_{0}$, depending on whether $q_{0}\leq (n-1)/2$ or $q_{0}\geq (n-1)/2$, respectively. In particular, the range of  allowable $q$'s is symmetric about $(n-1)/2$, in line with the expectation that estimates on the boundary should hold at symmetric form levels simultaneously (see the discussion around \eqref{5.33}). The estimates in Theorem \ref{main-result} are not independent: (3) holds for $(0,q)$--forms if and only if (2) holds at levels $(q-1)$, $q$, and $(q+1)$ (\cite{Harrington-Peloso-Raich-2015}), and (1) holds if and only if (2) holds (appropriately formulated when $q=0$ or $q=(n-1)$, \cite{Straube25}). 

\smallskip

We single out for emphasis the most important case in Theorem \ref{main-result}, that of comparable eigenvalues. Let $\lambda_1(p),\dots,\lambda_{n-1}(p)$ denote the eigenvalues of the Levi form of $b\Omega$ at $p$, listed with multiplicity. The eigenvalues are said to be comparable if there is a constant $c$ such that $\lambda_{j}(p)\leq c\lambda_{k}(p)$, for all $p\in b\Omega$ and $1\leq j,k\leq (n-1)$.
\begin{corollary}\label{comparable}
 Let $\O\subset\C^n$, with $n\geq 3$, be a smooth bounded pseudoconvex domain whose Levi eigenvalues are comparable. Then, if $DF(\Omega)=1$, the conclusions of Theorem \ref{main-result} hold for $u\in L^{2}_{(0,q)}(b\Omega)$, $0\leq q\leq (n-1)$.
\end{corollary}
Corollary \ref{comparable} follows from Theorem \ref{main-result} because the comparable eigenvalues condition implies $D(1)$, so $q_{0}=1$; see the discussion of $D(q)$ in section \ref{def}.

\smallskip

While $D(q)$ is indispensable in our current proof of Theorem \ref{main-result}, whether it is really needed is doubtful; the conjecture mentioned above that index one should imply regularity is plausible also for the complex Green operator. For example, on (smooth bounded) convex domains, the estimates (1)--(3) in Theorem \ref{main-result} always hold (by \cite{BoasStraube91b}), while $D(q)$ is rather restrictive. Indeed, a convex domain that satisfies $D(q)$ also satisfies $\widetilde{D}(q)$, hence satisfies maximal estimates on the interior (\cite{BenMoussa00, Derridj-1978}). As a result, its boundary must not contain analytic varieties of dimension $q$, except possibly ones that are contained in an $(n-1)$--dimensional subvariety (\cite{CSS20}, Corollary 1). 

\smallskip

The paper is organized as follows. Section \ref{def} gives the needed definitions and some preliminary background material. In section \ref{ICDA}, we discuss D'Angelo forms and their connections to the Diederich--Forn\ae ss index and to commutators with $\overline{\partial}_{b}$ and $\overline{\partial}_{b}^{*}$. We also recall an important estimate from \cite{Liu-Straube-2025}, which we then use in section \ref{main-estimate} to establish Proposition \ref{main-proposition}, for our purposes the principal consequence of having DF--index one. This $L^{2}$--type estimate is what ultimately lets us control commutators with $\overline{\partial}_{b}$ and $\overline{\partial}_{b}^{*}$. We use it in section \ref{proof} in combination with ideas from \cite{BoasStraube91b, Liu-Straube-2025, Straube25 } to prove Theorem \ref{main-result}. Section \ref{appendix} states a Kohn--Morrey--H\"{o}rmander type estimate for the boundary, a special case of an estimate in \cite{Harrington-Raich-2011}, simplified for what we need.

\section{Definitions and preliminaries}\label{def}
This section serves to recall some basic definitions and set up notation. Let $\Omega$ be a smooth bounded pseudoconvex domain in $\mathbb{C}^{n}$, with the CR--structure on $b\Omega$ induced by the ambient space. That is, $T^{1,0}(b\Omega)$ is the bundle of vectors $\sum_{j=1}^{n}a_{j}(z)\partial/\partial z_{j}$ with $\sum_{j=1}^{n}a_{j}\partial\rho/\partial z_{j} =0$ for some defining function $\rho$ (hence for any defining function). For $0\leq q\leq (n-1)$, denote by $\overline{\partial}_{b}: L^{2}_{(0,q)}(b\Omega)\rightarrow L^{2}_{(0,q+1)}(b\Omega)$ the usual (extrinsically defined) tangential Cauchy--Riemann operator, and by $\overline{\partial}_{b}^{*}$ its Hilbert space adjoint. These operators have closed range, and there is an associated Hodge theory (\cite{Shaw85b, Boas-Shaw-1986, Kohn86, ChenShaw01, Biard-Straube-2017}):
\begin{equation}\label{Hodge}
 L^{2}_{(0,q)}(b\Omega) = \underbrace{\ker(\overline{\partial}_{b})}_{\im(\overline{\partial}_{b})\;q=(n-1)} \oplus \underbrace{\ker(\overline{\partial}_{b}^{*})}_{\im(\overline{\partial}_{b}^{*})\;q=0}\;;\;1\leq q\leq (n-2)\;.
\end{equation}
The underbraces indicate the modifications for the respective values of $q$. In particular,
\begin{equation}\label{basic}
 \|u\|^{2} \leq C(\|\overline{\partial}_{b}u\|^{2}+\|\overline{\partial}_{b}^{*}u\|^{2})\;;\;u \in \dom(\overline{\partial}_{b})\cap\dom(\overline{\partial}_{b}^{*})\;,
\end{equation}
with $u\in \ker(\overline{\partial}_{b})^{\perp}$ when $q=0$, and $u \in \ker(\overline{\partial}_{b}^{*})^{\perp}$ when $q=(n-1)$.

For $1\leq q\leq (n-2)$, the operator $\Box_{b,q}:=\overline{\partial}_{b}\overline{\partial}_{b}^{*} + \overline{\partial}_{b}^{*}\overline{\partial}_{b}$ is selfadjoint, injective, and onto, hence has a bounded inverse $G_{q}$. This is the complex Green operator. We then have  for $u\in L^{2}_{(0,q)}(b\Omega)$, $1\leq q\leq (n-2)$:
\begin{equation}\label{BoxG}
 u = \overline{\partial}_{b}\overline{\partial}_{b}^{*}G_{q}u + \overline{\partial}_{b}^{*}\overline{\partial}G_{q}u\;.
\end{equation}
If  $u$ is $\overline{\partial}_{b}$--closed, \eqref{BoxG} implies that $\overline{\partial}_{b}^{*}\overline{\partial}G_{q}u$ is also $\overline{\partial}_{b}$--closed; since it is also orthogonal to $\ker(\overline{\partial}_{b})$, it must vanish, and $u=\overline{\partial}_{b}(\overline{\partial}_{b}^{*}G_{q}u)$. That is, $\overline{\partial}_{b}^{*}G_{q}$ is a solution operator for $\overline{\partial}_{b}$. It gives the unique solution of minimal $L^{2}$--norm (also referred to as the canonical or Kohn solution). The $G_{q}$ commute with $\overline{\partial}_{b}$ and $\overline{\partial}_{b}^{*}$ (with appropriate form levels). The Szeg\"{o} projection $S_{q}$ is the orthogonal projection from $L^{2}_{(0,q)}(b\Omega)$ onto $\ker(\overline{\partial}_{b})$ when $0\leq q\leq (n-2)$, and onto $\im(\overline{\partial}_{b})=\ker(\overline{\partial}_{b}^{*})^{\perp}$ when $q=(n-1)$. We then have $S_{q}=\overline{\partial}_{b}\overline{\partial}_{b}^{*}G_{q}$ for $1\leq q\leq (n-2)$, and $S_{q} = I - \overline{\partial}_{b}^{*}G_{q+1}\overline{\partial}_{b}\;;\;0\leq q\leq (n-3)$, where $I$ is the identity.

\smallskip

Our proof of Theorem \ref{main-result} uses maximal estimates for the boundary complex, i.e.,
\begin{equation}\label{maxest}
    \norm{Yu}_{k}^2 \le C\left( \norm{\dbarb u}^2_{k} + \norm{\dbarbstar u}^2_{k} \ + \norm{u}^2_{k}\right)\;,
\end{equation}
where $Y$ is a complex tangential vector field. This estimate is stronger than the standard estimate available for complex tangential derivatives (\cite{BoasStraube91b}, Lemma 1):
\begin{equation}\label{comptang}
    \norm{Yu}_{k}^2 \le C\left( \norm{\dbarb u}^2_{k} + \norm{\dbarbstar u}^2_{k}  +\norm{u}_{k}\norm{u}_{k+1}\right)\;.
\end{equation}
The necessary and sufficient condition for the maximal estimates \eqref{maxest} to hold at the level of $(0,q)$--forms ($1\leq q\leq (n-2)$) is condition $D(q)$. For $1\le q\le n-1$, let $\sigma_q(p)$ denote any of the ${{n-1}\choose{q}}$ sums $q$ the eigenvalues and let $\tau(p) = \lambda_1(p) + \cdots + \lambda_{n-1}(p)$ denote the trace of the Levi form. Following \cite{Koenig-2002}, we say that condition $D(q)$ holds (for $1\leq q \leq n-2$), if there exists $\epsilon>0$ such that
\begin{equation}\label{condition-dq-inequality}
    \epsilon\tau(p) \le \sigma_q(p) \le (1-\epsilon)\tau(p),
\end{equation}
for all $p\in b\Omega$ and any $q$-sum $\sigma_q$ of the eigenvalues of the Levi form. It was shown in \cite{Derridj-1991} for $q=1$ and in \cite{Koenig-2002} for general $q$ that condition $D(q)$ is a sufficient condition for maximal estimates for $\dbarb\oplus\overline{\partial}_{b}^{*}$ to hold at form level $(0,q)$. According to \cite{Derridj-1991}, Remarque 2 on page 633, it is also necessary. For $1\le q_0\le n-2$, condition $D(q_0)$ implies condition $D(q)$ for $\min\{q_0,n-1-q_0\} < q < \max\{q_0,n-1-q_0\}$ (see Proposition A.3 in \cite{Koenig-2002}). For this reason, the condition is sometimes referred to as symmetric condition $D(q)$, in contrast to the nonsymmetric condition $\widetilde{D}(q)$: $\varepsilon\tau \leq \sigma_{q}$, which is the necessary and sufficient condition for maximal estimates on the interior (\cite{Derridj-1978, BenMoussa00}). Condition $D(1)$ never holds in $\mathbb{C}^{2}$, while $\widetilde{D}(1)$ always does. On the other hand, when $n\geq 3$, it is easy to see that $D(1)$ and $\widetilde{D}(1)$ are equivalent (and that $D(q)$ and $\widetilde{D}(q)$ are not when $q>1$). In particular, in the most important case, that of comparable eigenvalues, $D(1), \cdots, D(n-2)$ are all satisfied. It is pointed out in \cite{Koenig-2002} that in the pseudoconvex case, the classical condition $Y(q)$ is equivalent to $0<\sigma_{q}<\tau$, uniformly over the boundary. By a continuity argument, this latter condition implies $D(q)$, which may thus be seen as a generalization of $Y(q)$ (for pseudoconvex domains).

\smallskip

In 1977, Diederich and Forn\ae ss proved that for every smooth bounded pseudoconvex domain in $\mathbb{C}^{n}$, there is $\eta>0$ and a defining function $\rho_{\eta}$ such that $-(-\rho_{\eta})^{\eta}$ is plurisubharmonic on $\Omega$ (near the boundary; having this requirement only near the boundary does not affect the definition of the index, see \cite{Harrington22}, Lemma 2.3). It is not hard to see that the set of such $\eta$ is an interval bounded below by zero. The supremum over such $\eta$, $0<\eta\leq 1$, is called the Diederich-Forn\ae ss index of $\Omega$; we denote it by $DF(\Omega)$. Because $DF(\Omega)$ is a supremum, index one does not imply the existence of a defining function that is plurisubharmonic at the boundary. The reader is referred to Section 2.3 in \cite{Straube25b} for further discussion of the index. It is noteworthy that in all cases where the index is known to be one, it is also known that regularity holds for both the $\overline{\partial}$--Neumann operator and the complex Green operator. 

\section{DF--index, commutators, and D'Angelo forms}\label{ICDA}

The link between the index and (commutator) estimates is provided by D'Angelo forms. Their role in estimating commutators is well known. Let $\rho$ be a defining function for $\Omega$. Set $L_{n,\rho}=(1/|\partial\rho|^{2})\sum_{j=1}^{n}(\partial\rho/\partial \overline{z_{j}})\partial/\partial z_{j}$ (so that $L_{n,\rho}\rho\equiv 1$), $T_{\rho}=L_{n,\rho}-\overline{L_{n,\rho}}$, and $\sigma_{\rho}=(1/2)(\partial\rho-\overline{\partial}\rho)$ (so that $\sigma_{\rho}T_{\rho}\equiv 1$). Set
\begin{equation}\label{alpha}
 \alpha_{\rho} = -\mathcal{L}_{T_{\rho}}\sigma_{\rho}\;,
\end{equation}
where $\mathcal{L}$ denotes the Lie derivative. Then for $L\in T^{1,0}(b\Omega)$, $\alpha_{\rho}(L)=\sigma_{\rho}([T_{\rho},L])$ and $\alpha_{\rho}(\overline{L})=\sigma_{\rho}([T_{\rho},\overline{L}])$ (note that $\sigma_{\rho}(T_{\rho})\equiv 1$). The form $\alpha_{\rho}$ is called the D'Angelo form associated to $\rho$. Since both $T_{\rho}$ and $\sigma_{\rho}$ are purely imaginary, $\alpha_{\rho}$ is real. We will need to know how $\alpha_{\rho}$ changes when the defining function changes. If $e^{h}\rho$ is another defining function, a computation gives
\begin{equation}\label{changealpha}
 \alpha_{e^{h}\rho}(X) = \alpha_{\rho}(X) + dh(X)\;;\;X\in T^{1,0}(b\Omega)\oplus T^{0,1}(b\Omega)\;;
\end{equation}
see \cite{Straube-2010}, Section 5.9, \cite{Dall'Ara-Mongodi-2021}, Section 4, and \cite{Straube25b}, section 2.2 for more information on D'Angelo forms. It will be important later that various estimates can be done with constants that are uniform over the boundaries of approximating subdomains, that is over $b\Omega_{\delta}$, where $\Omega_{\delta}=\{z\in \Omega\,|\,\rho(z)+\delta<0\}$, and $0<\delta<\delta_{\rho}$. There are two things to note in this context. First, $L_{n,\rho}$, $T_{\rho}$, $\sigma_{\rho}$ and $\alpha_{\rho}$ are defined in a one sided neighborhood of $b\Omega$. Second, at points of $b\Omega_{\delta}$, they agree with the corresponding notions on $b\Omega_{\delta}$ (with $\rho+\delta$ as defining function, then $\nabla(\rho)=\nabla(\rho+\delta)$).

\smallskip

When computing $\overline{\partial}_{b}$ and $\overline{\partial}_{b}^{*}$ in a special boundary chart, the derivatives that occur are are of the form $\overline{L}$ or $L$, with $L\in T^{1,0}(b\Omega)$. So $\alpha_{\rho}(L)=\sigma_{\rho}([T_{\rho},L])$ and $\alpha_{\rho}(\overline{L})=\sigma_{\rho}([T_{\rho},\overline{L}])$ pick out the $T_{\rho}$--component (the bad direction) of the commutators with $T_{\rho}$. Denote by $L_{1},\cdots, L_{n}$ a special boundary frame, as usual, with dual frame $\omega_{1},\cdots, \omega_{n}$. If $v=\sum^{'}_{|J|=q}v_{J}\overline{\omega_{J}}$ is a $(0,q)$--form supported in this chart, then letting $T_{\rho}$ act coefficientwise in this chart and also computing $\overline{\partial}_{b}$ and $\overline{\partial}_{b}^{*}$ in this chart gives the following two commutator formulas when $k=1$:
 \begin{equation}\label{commutator-formula-dbarb}
            [\dbarb,(T_\rho)^k]v = -k\sum_{j=1}^{n-1}\isum_{J} \alpha^\rho(\overline{L}_j)((T_\rho)^kv_J)(\overline{\omega_{j}}\wedge\overline{\omega_{J}}) +  A_h(v) + B_h(v)
    \end{equation}
    and for $q\ge 1$ ,
    \begin{equation}\label{commutator-formula-dbarbstar}
            [\dbarbstar,(T_\rho)^k]v = k\sum_{j=1}^{n-1}\isum_{S} \alpha^\rho(L_j)((T_\rho)^kv_{jS})\overline{\omega_{S}} + \tilde{A}_h(v) + \tilde{B}_h(v),
    \end{equation}
where the terms $A_h(v)$ and $\tilde{A}_h(v)$ consist of terms of order $k$ with at least one of the derivatives being complex tangential, and the terms $B_h(v)$ and $\tilde{B}_h(v)$ consist of terms of order at most $(k-1)$. Induction on $k$ then gives the general case. These arguments are completely analogous to the ones that give the corresponding formulas for the commutators with $\overline{\partial}$ and $\overline{\partial}^{*}$ (\cite{Harrington-Liu-2020}, Lemmas 3.1, 3.2; \cite{Liu-Straube-2025}, Lemmas 3, 4; compare also Proposition 3.7 in \cite{LiuRaich20}).                                                                                                                                                                                                                                   
                                                                                                                                                                                    
\smallskip

While the significance of D'Angelo forms for commutator estimates has been known for a long time, their relevance for the Diederich--Forn\ae ss index is a recent discovery, due to Liu (\cite{Liu19}) and Yum (\cite{Yum-2021}). 
Consider $\eta\in (0,1)$ such that there exists a defining function $\rho_{\eta}$ whose D'Angelo form $\alpha_{\rho_{\eta}}$ satisfies
\begin{equation}\label{alpha-index}
 |\alpha_{\rho_{\eta}}(L)(p)|^{2} \le 2\frac{(1-\eta)}{\eta}\overline{\partial}\alpha_{\rho_{\eta}}(L,\overline{L})(p)\;;\;L\in\mathcal{N}_{p}\;,\;p\in b\Omega\;,
\end{equation}
where $\mathcal{N}_{p}$ denotes the null space of the Levi form at $p\in b\Omega$. Then
\begin{equation}\label{index2}
 DF(\Omega) = \sup\{\eta\in (0,1)\,|\,\exists \;\text{a defining function}\;\rho_{\eta}\;\text{such that}\;\eqref{alpha-index}\; \text{holds}\}\;;
\end{equation}
see \cite{Liu19}, Theorem 2.10, \cite{Yum-2021}, Theorem 1.1; 
further discussion is in \cite{Straube25b}, Section 4. Note that the set of $\eta$ with a defining function $\rho_{\eta}$ such that \eqref{alpha-index} holds is also an interval (because $\frac{1-\eta}{\eta}$ is a decreasing function). In \cite{Liu-Straube-2025}, the authors show the following. Let $0<\eta<DF(\Omega)$, with associated defining function $\rho_{\eta}=e^{h_{\eta}}\rho$ such that \eqref{alpha-index} holds. There is a constant $C$ independent of $\eta$ and a constant $M_{\eta}$ such that 
\begin{multline}\label{alphaest1}
 |\alpha_{\rho_{\eta}}(L)|^{2} \leq C(1-\eta)\left(\sum_{j,k}\frac{\partial^{2}(-h_{\eta})}{\partial z_{j}\partial\overline{z_{k}}}L_{j}\overline{L_{k}} + 2|L|^{2}\right) \\
 + M_{\eta}\sum_{j,k}\frac{\partial^{2}\rho}{\partial z_{j}\partial\overline{z_{k}}}L_{j}\overline{L_{k}}\;\;;\;L=\sum_{j}L_{j}\partial/\partial z_{j} \in T^{1,0}(b\Omega)
 \,.
\end{multline}
Moreover, there is $\delta_{\eta}>0$ such that this estimate holds uniformly on the level sets $\{\rho_{\eta}=-\delta\}$ for $\delta<\delta_{\eta}$.  Roughly speaking, the Hessian of $-h_{\eta}$ appears because $\overline{\partial}\alpha_{\rho_{\eta}}$ in \eqref{alpha-index} equals $\overline{\partial}(dh_{\eta}+\alpha_{\rho})$ (some care is needed though; see \cite{Dall'Ara-Mongodi-2021}, Section 4.3, especially Lemmas 4.6 and 4.7, for details), and the $M_{\eta}$--term is needed for when $L(p)\notin \mathcal{N}_{p}$ (see \cite{Liu-Straube-2025}, section 3). In order to obtain an estimate for forms from \eqref{alphaest1}, we define auxiliary vector fields associated to forms as follows. For a $(0,q)$--form $u=\sideset{}{'}\sum_{|K|=q}u_{K}d\overline{z_{K}}$ and an increasing $(q-1)$--tuple $J$, set $L^{J}_{u}=\sum_{j}u_{jJ}\partial/\partial z_{j}$ (when $q=1$, we understand $J$ to be the empty tuple, and $L_{u}=\sum_{j}u_{j}\partial/\partial z_{j}$). When $u$ is a form on the boundary, then for all $J$, $L^{J}_{u}\rho = \sum_{j}u_{jJ}(\partial\rho/\partial z_{j}) = 0$, so that \eqref{alphaest1} does apply to $L^{J}_{u}$. In sections \ref{main-estimate} and \ref{proof}, we will work in special boundary charts, with the usual (local) orthonormal basis $L_{1}, L_{2},\cdots, L_{n}$ for vetcor fields of type $(1,0)$, and dual basis $\omega_{1},\cdots, \omega_{n}$. If $u=\sum^{'}u_{K}\overline{\omega_{K}}$, the analogous construction in the boundary chart is $\tilde{L}_{u}^{J}=\sum_{j}u_{jJ}L_{j}$. Applying \eqref{alphaest1} to $\tilde{L}_{u}^{J}$, and summing over $J$, gives
\begin{multline}\label{alphaest2}
 \sideset{}{'}\sum_{J}|\alpha_{\rho_{\eta}}(\tilde{L}^{J}_{u})|^{2} \leq C(1-\eta)\left(\sideset{}{'}\sum_{J}\partial\overline{\partial}(-h_{\eta})\big(\tilde{L}_{u}^{J}\wedge \overline{\tilde{L}_{u}^{J}}\big) + 2\|u\|^{2}\right)\\ 
 + M_{\eta}\sideset{}{'}\sum_{J}\sum_{j,k}c_{jk}u_{jJ}\overline{u_{kJ}}\;\;,
\end{multline}
where $(c_{jk})$ are the coefficients of the Levi form in the boundary chart.
There is $\delta_{\eta}$ so that this estimate holds uniformly over the level sets $\rho_{\eta}=-\delta$ for $0\leq \delta<\delta_{\eta}$ (inherited from \eqref{alphaest1}). The above construction is not invariant (unless $q=1$). However, the quantity that will matter, namely $\sum^{'}_{J}|\alpha_{\rho_{\eta}}(L_{u}^{J})|^{2}$ (respectively $\sum^{'}_{J}|\alpha_{\rho_{\eta}}(\tilde{L}_{u}^{J})|^{2}$) is, up to a constant factor. The reason is that each $L_{u}^{J}$ is a linear combination, with smooth coefficients independent of $u$, of the $\tilde{L}_{u}^{J}$, and vice versa (this observation from \cite{Liu-Straube-2025}, page 5417, is based on the transformation formulas between the two orthonormal frames). Estimate \eqref{alphaest2} is the starting point, in the next section, for deriving the crucial $L^{2}$--type estimate needed in the proof of Theorem \ref{main-result}.

\section{The main estimate}\label{main-estimate}

To simplify notation from section \ref{ICDA}, now that there is an additional subscript $\eta$, we replace the double subscripts $\rho_{\eta}$ by the simple subscript $\eta$. That is  $T_{\rho_{\eta}}$ becomes $T_{\eta}$, $\alpha_{\rho_{\eta}}$ becomes $\alpha_{\eta}$, and so on. The estimate in Proposition \ref{main-proposition} is, for our purposes, the principal consequence of having DF--index 1.
\begin{proposition}\label{main-proposition}
Let $\O\subset\C^n$ be a bounded pseudoconvex domain with $DF(\Omega)=1$, and satisfying condition $D(q)$ for some $1\le q\le n-2$. There are constants $C$ and $\eta_{0}$, $0<\eta_{0}<1$, with the following property. For $\eta_{0}<\eta<1$, there are constants $C_{\eta}$ and $\delta_{\eta}>0$ such that for all $u\in C^\infty_{(0,q)}(b\Omega)$, we have
\begin{equation}\label{estimate-main-proposition-euclidean}
    \isum_{|J|=q-1}\int_{b\Omega}\left|\alpha_{\eta}(L_u^J) \right|^2 \le C(1-\eta)\left( \|\dbar_b u\|^2 + \|\dbar_b^* u\|^2 \right) + C_\eta\|u\|_{-1}^2\,,
\end{equation}
and, moreover, this estimate holds on the level sets $\{\rho_{\eta}=-\delta\}$ for $\delta<\delta_{\eta}$ with the same constants.
\end{proposition}
\begin{proof}
The uniformity part will follow by keeping track of the constants and, where partitions of unity over $b\Omega$ and boundary charts are used, using instead partitions of unity and charts over a neighborhood of $b\Omega$. 

\smallskip
It suffices to prove the proposition for forms supported in suitable small patches of the boundary (independent of $\eta$), the general case then follows via a partition of unity.
In view of the discussion at the end of the previous section, we may replace the $L_{u}^{J}$ in \eqref{estimate-main-proposition-euclidean} by $\tilde{L}_{u}^{J}$. For $u$ supported in a special boundary chart as in Proposition \ref{proposition-kohn-morrey-hormander} in section \ref{appendix}, we look at the term involving the Hessian of $-h_{\eta}$ on the right hand side of \eqref{alphaest2}. Proposition \ref{proposition-kohn-morrey-hormander}, with $\varphi = -h_{\eta}$, and $(e^{-h_{\eta}}/2)u$ in place of $u$, gives
\begin{multline}\label{prop-eq1}
         \normwh{\overline{L}(e^{-h_\eta/2}u)}^2  + 2\text{Re} \isum_{K}\sum_{j,k}\int_{b\O}  c_{jk}\left(\left( T + \frac{1}{2}\frac{\partial h_\eta}{\partial\nu} \right) u_{jK}\right)\overline{u_{kK}}  \\
         + 2\isum_{K}\int_{b\O} \partial\overline{\partial}(-h_{\eta})\big(\tilde{L}_{u}^{K}\wedge\overline{\tilde{L}_{u}^{K}}\big)  +  \bigo{\norm{u}^2} \\
        \lesssim \normwh{\dbarb (e^{-h_\eta/2}u)}^2 + \normwh{\dbarbstarwh (e^{-h_\eta/2}u)}^2\;;
\end{multline}
here $\overline{L}(\cdot)$ denotes the vector of all $\overline{L_{j}}$--derivatives and the constant in the $\mathcal{O}(\|u\|^{2})$ is independent of $\eta$. We rewrite \eqref{prop-eq1} as
\begin{multline}\label{prop-eq2}
\isum_{K}\int_{b\O} \partial\overline{\partial}(-h_{\eta})\big(\tilde{L}_{u}^{K}\wedge\overline{\tilde{L}_{u}^{K}}\big)
           \lesssim \normwh{\dbarb (e^{-h_\eta/2}u)}^2 + \normwh{\dbarbstarwh (e^{-h_\eta/2}u)}^2  \\
          - 2\text{Re} \isum_{K}\sum_{j,k}\int_{b\O}  c_{jk}\big((T+\frac{1}{2}\frac{\partial h_{\eta}}{\partial\nu}) u_{jK}\big)\overline{u_{kK}} + \bigo{\norm{u}^2}\;.
\end{multline}
We have used that $\normwh{\overline{L}(e^{-h_\eta/2}u)}^2 \ge 0$. We now apply the estimate \eqref{prop-eq2} to \eqref{alphaest2} to obtain
\begin{multline}\label{prop-eq3}
        \isum_{K} \int_{b\O} \abs{\alpha_\eta(\tilde{L}_{u}^K)}^2  \lesssim (1-\eta)\Big( \normwh{\dbarb (e^{-h_\eta/2}u)}^2 + \normwh{\dbarbstarwh (e^{-h_\eta/2}u)}^2  \\
         \;\;\;\;\;\;\;\;\;\;\;\;\;\;\;\;\;\;\;\;\;\;\;\;\;- 2\text{Re} \isum_{K}\sum_{j,k}\int_{b\O}  c_{jk}\big((T+\frac{1}{2}\frac{\partial h_{\eta}}{\partial\nu}) u_{jK}\big)\overline{u_{kK}} + \bigo{\norm{u}^2} \Big)  \\
         + M_\eta \isum_{K} \sum_{j,k} \int_{b\O}c_{jk}u_{jK}\overline{u_{kK}}.
\end{multline}
Choose $\tilde{M}_{\eta}$ big enough so that $\tilde{M}_\eta \geq M_\eta/(1-\eta)-(1/2)(\partial h_{\eta}/\partial\nu)$. Then \eqref{prop-eq3} becomes
\begin{multline}\label{prop-eq3a}
        \isum_{K} \int_{b\O} \abs{\alpha_\eta(\tilde{L}_{u}^K)}^2  \lesssim (1-\eta)\bigg( \normwh{\dbarb (e^{-h_\eta/2}u)}^2 + \normwh{\dbarbstarwh (e^{-h_\eta/2}u)}^2 \\
         - 2\text{Re}\big( \isum_{K}\sum_{jk} \int_{b\O} c_{jk} \big((T - \tilde{M}_\eta) u_{jK}\big)\overline{u_{kK}}\, \big)
        +\bigo{\norm{u}^2}\bigg)
\end{multline}
(since the Levi from is positive semidefinite). 

\smallskip

To proceed further, we may assume that $u$ is supported in a local boundary chart where we have the usual microlocalizations, which we now recall (see \cite{Kohn85, Kohn-2002, Biard-Straube-2019, MunasingheStraube12}). Assume that $U'$ is an open subset of $b\O$ contained in a special boundary chart. We will work with forms supported in a fixed open set $U \subset\subset U'$ small enough so that the following makes sense. Choose coordinates on $b\O$ in $U'$ of the form $(x_1, \dots, x_{2n-2}, t)$
such that $T = (-i)\frac{\partial}{\partial t}$. Denote the `dual' coordinates (on the Fourier transform side) in $\R^{2n-1}$ by $(\xi_1,\dots,\xi_{2n-2},\tau) = (\xi,\tau)$. Next, choose $\chi\in C^\infty_0(U')$ with $\chi\equiv 1$ in a neighborhood of $\overline{U}$. On the unit sphere $\{\norm{\xi}^2 + \abs{\tau}^2 = 1 \}$, choose a smooth function $g$ with range in $[0,1]$ and satisfying the following property: $\text{supp}(g)\subseteq \{ \tau > \norm{\xi}/2 \} \;\text{with}\, g\equiv 1 \;\text{on}\;\{ \tau \ge 3\norm{\xi}/4\}\;$. For $|(\xi,\tau)|\geq 3/4$, we set $\chi^{+}(\xi,\tau)=g(\frac{(\xi,\tau)}{|(\xi,\tau)|})$ and extend it smoothly so that $\chi^{+}(\xi,\tau) = 0$ when $|(\xi,\tau)|\leq 1/2$. As usual, we set $\chi^{-}(\xi,\tau)=\chi^{+}(-\xi,-\tau)$, and $\chi^{0}=1-\chi^{+}-\chi^{-}$, with associated pseudodifferential operators $\mathcal{P}^{+}$, $\mathcal{P}^{-}$, and $\mathcal{P}^{0}$. This is the standard set up which we  now modify. The idea is to restrict the support of $\chi^{+}$ and $\chi^{-}$ so that $|\tau|$ is sufficiently big. Let $\tilde{M}_\eta >0$ be the constant from \eqref{prop-eq3a}. Denote be $\sigma$ a smooth nondecreasing function on $\mathbb{R}$ with $\sigma(x)\equiv 0$, $x\leq 1$, and $\sigma(x)\equiv 1$, $x\geq 2$. We set
\begin{equation}\label{chi-plus-definition}
\chi^{+,\eta} = \chi^{+}\sigma(\tau-3\tilde{M}_{\eta})\;,\;\chi^{-,\eta} = \chi^{-}\sigma(-\tau-3\tilde{M}_{\eta})\;,\;\text{and}\;\chi^{0,\eta}=1 - \chi^{+,\eta} - \chi^{-,\eta}\;. 
\end{equation}
Then $\chip$ and $\chim$ are supported in $\{(\xi,\tau) : \tau \ge 1 + 3\tilde{M}_{\eta}\}$ and $ \{(\xi,\tau) : \tau \leq -1 - 3\tilde{M}_{\eta}\}$, respectively, and
$\chiz$ is supported in an elliptic region that also depends on $\eta$. It will be important later that these modified symbols differ from the standard ones by compactly supported terms (albeit with support depending on $\eta$). Denote the Fourier transform on $\R^{2n-1}$ by $\F$. For a $(0,q)$-form $u$ supported in $U$, we set
\begin{equation}\label{definition-microlocal-operators}
    \pplus u = \chi \F^{-1}(\chip \F u), \quad \pminus u = \chi \F^{-1}(\chim \F u), \quad \pzero u = \chi \F^{-1}(\chiz \F u),
\end{equation}
where the operators act coefficient-wise with respect to the frame $\{\omega_1,\dots,\omega_{n-1}\}$. Because of the factor $\chi$, the forms in \eqref{definition-microlocal-operators} 
can be thought of as forms on the whole boundary. For $u\in L^{2}_{(0,q)}(b\Omega)$, supported in $U$, we have the decomposition:
\begin{equation}\label{microlocal-decomposition}
    u = \pplus u + \pminus u + \pzero u\;.
\end{equation}
It now suffices to prove that $\sum^{'}_{|J|=q-1}\int_{b\Omega}|\alpha_{\eta}(\tilde{L}_{\mathcal{P}^{k,\eta}u}^{J})|^{2}$ is dominated by the right hand side of \eqref{estimate-main-proposition-euclidean}, for $k=+,0,-$.

\smallskip

We start with $\mathcal{P}^{+,\eta}u$ and first estimate the Re--term in the second line of \eqref{prop-eq3a} for $\mathcal{P}^{+,\eta}u$ from below. We have
\begin{equation}\label{prop-eq4}
    \begin{split}
        \isum_{K}\sum_{j,k} & \int_{b\O} c_{jk} \left(( T - \tilde{M}_\eta) (\pplus u)_{jK}\right)\overline{(\pplus u)_{kK}} \\ 
        = & \isum_{K}\sum_{j,k} \int_{b\O} c_{jk} \left( ( T - \tilde{M}_\eta ) (\chi\F^{-1} (\chip \F u_{jK})) \right) \overline{\chi\F^{-1} (\chip \F u_{kK})}.
    \end{split}
\end{equation}
Commuting $\chi$ with $T$, we see that the right hand side in \eqref{prop-eq4} equals
\begin{multline}\label{prop-eq5}
\isum_{K}\sum_{j,k} \int_{b\O} c_{jk} \chi\left( ( T - \tilde{M}_\eta ) ( \F^{-1} (\chip \F  u_{jK})) \right) \overline{(\chi \F^{-1} (\chip \F  u_{kK}))}
        \\  + \isum_{K}\sum_{j,k} \int_{b\O} c_{jk} \left( [T,\chi] ( \F^{-1} (\chip \F u_{jK})) \right) \overline{(\chi \F^{-1} (\chip \F u_{kK}))}.
    \end{multline}
Using the fact that $[T,\chi]$ is of order zero and the Plancherel theorem, the last term in \eqref{prop-eq5} can be estimated by $\|u\|^{2}\lesssim \|\overline{\partial}_{b}u\|^{2}+\|\overline{\partial}_{b}^{*}u\|^{2}$, with a constant that does not depend on $\eta$.  Now consider the first term in \eqref{prop-eq5}. The following calculations are similar to those in Lemma 2.5 in \cite{Kohn-2002}. Denote by $\tilde{\chi}^{+}$ a smooth function with values in $[0,1]$ that is supported in $\{\tau \geq \|\xi\|/3\}$, and that is identically one in a neighborhood of the support of $\chi^{+}$ (constructed analogously to $\chi^{+}$, with a suitable function $\tilde{g}$). Then set $\tilde{\chi}^{+,\eta}=\tilde{\chi}^{+}\sigma(\tau-3\tilde{M}_{\eta}+1)$. Note that $\tilde{\chi}^{+,\eta}$ is identically one on the support of $\chi^{+,\eta}$ and on the support of $\tilde{\chi}^{+,\eta}$, we have $\tau-\tilde{M}_\eta \ge 2\tilde{M}_\eta \ge 1$ (we may assume that $\tilde{M}_\eta \ge 1/2$). Denote by $R^\eta$ the pseudodifferential operator of order $1/2$ whose symbol is $(\tau - \tilde{M}_\eta)^{1/2}\chipt$. Since $\chipt\equiv 1$ on $\supp(\chip)$, the first term in \eqref{prop-eq5} can be written, after switching to integration over $\mathbb{R}^{2n-1}$, as
\begin{multline}\label{prop-eq7}
 \isum_{K}\sum_{j,k} \int_{\mathbb{R}^{2n-1}}\mu c_{jk} \chi\F^{-1}\left( (\tau - \tilde{M}_\eta) (\chipt)^2 \chip \F u_{jK}\right) \overline{(\chi \F^{-1} (\chip \F u_{kK}))} \\
          = \isum_{K}\sum_{j,k} \int_{\mathbb{R}^{2n-1}}\mu\chi^{2} c_{jk} (R^\eta)^2\F^{-1}\left( \chip \F u_{jK}\right) \overline{ \F^{-1} (\chip \F  u_{kK})}\;,
\end{multline}
where $\mu$ denotes the density of the the surface measure on $b\Omega$. We have used that the symbol of $R^{\eta}$ is independent of the $(x_{1},\cdots,x_{2n-2},t)$ variables. Taking real parts and using that $R^{\eta}$ is selfadjoint, we have
\begin{multline}\label{prop-eq8}
 \text{Re}\isum_{K}\sum_{j,k} \int_{\mathbb{R}^{2n-1}} \mu\chi^{2} c_{jk} (R^\eta)^2\F^{-1}\left( \chip \F u_{jK}\right) \overline{ \F^{-1} (\chip \F  u_{kK})} \\
      = \text{Re}\bigg( \isum_{K}\sum_{j,k} \int_{\mathbb{R}^{2n-1}} [\mu \chi^{2}c_{jk}, R^\eta] R^\eta\F^{-1}\left( \chip \F u_{jK}\right) \overline{ \F^{-1} (\chip \F  u_{kK})} \\
        + \isum_{K} \sum_{j,k} \int_{\mathbb{R}^{2n-1}} \mu\chi^{2} c_{jk}  R^\eta\F^{-1}\left( \chip \F u_{jK}\right) \overline{R^\eta \F^{-1} (\chip \F  u_{kK})}\;\bigg) \\
        \ge \text{Re}\isum_{K}\sum_{j,k} \int_{\mathbb{R}^{2n-1}} [\mu\chi^{2} c_{jk}, R^\eta] R^\eta\F^{-1}\left( \chip \F u_{jK}\right) \overline{ \F^{-1} (\chip \F u_{kK})}.
\end{multline}
For the last inequality we have used pseudoconvexity of $\O$ and $\mu\chi^{2}\geq 0$ to see that the term in the third line of \eqref{prop-eq8} is nonnegative. We need a lower bound on the right hand side of \eqref{prop-eq8}. Such a bound will follow from an estimate on the absolute value of the expression inside Re.

\smallskip

Since $\chi R^\eta$ is of order $1/2$, the operator $[\mu c_{jk}\chi_s,\chi R^\eta]$ is of order $-1/2$, so $[\mu c_{jk}\chi,\chi R^\eta]R^\eta$ is of order zero (see \cite{Stein93}, Theorem 2 on page 237 for the relevant symbolic calculus). However, the constant this argument gives in the $L^{2}$--estimate depends on $\eta$, and further analysis is required. Modulo a symbol in $S^{-3/2}$, the symbol of $[\mu c_{jk}\chi^{2}, R^\eta]$ is
\begin{equation}\label{prop-eq-9}
    \begin{split}
        \sum_{|\alpha|=1} &\dfrac{1}{2\pi i} \bigg ( \partial^\alpha_{\tau,\xi}(\mu c_{jk}\chi) \,\partial^\alpha_{t,x}\left((\tau-\tilde{M}_\eta)^{1/2}\chipt(\tau,\xi)\right) 
        \\ & \qquad\qquad  - \partial^\alpha_{\tau,\xi}\left((\tau-\tilde{M}_\eta)^{1/2}\chipt(\tau,\xi)\right) \partial^\alpha_{t,x}(\mu c_{jk}\chi)\bigg ),
    \end{split}
\end{equation}
where $\alpha$ is a multi-index of order $1$ and $\partial^\alpha$ denotes partial derivatives. Since $\partial^\alpha_{\tau,\xi}(\mu c_{jk}\chi)=0$, \eqref{prop-eq-9} becomes
\begin{equation}\label{prop-eq10}
    -\dfrac{1}{2\pi i} \left( \dfrac{\chipt(\tau,\xi)}{2(\tau-\tilde{M}_\eta)^{1/2}} + (\tau-\tilde{M}_\eta)^{1/2}\partial_{\tau,\xi}\chipt(\tau,\xi) \right) \partial_x(\mu c_{jk}\chi).
\end{equation}
So modulo a symbol in $S^{-1}$, the symbol of $[\mu c_{jk}\chi^{2}, R^{\eta}]R^\eta$ is
\begin{equation}\label{prop-eq11}
    \begin{split}
        &-\dfrac{1}{2\pi i} \left( \dfrac{\chipt(\tau,\xi)}{2(\tau-\tilde{M}_\eta)^{1/2}} + (\tau-\tilde{M}_\eta)^{1/2}\partial_{\tau,\xi}\chipt(\tau,\xi) \right)\partial_{t,x}(\mu c_{jk}\chi) (\tau - \tilde{M}_\eta)^{1/2}\chipt(\tau,\xi) 
       \\ & = -\dfrac{1}{2\pi i} \left( \frac{1}{2} (\chipt(\tau,\xi))^2\partial_{t,x}(\mu c_{jk}\chi)  + (\tau-\tilde{M}_\eta)\chipt(\tau,\xi)\partial_{\tau,\xi}\chipt(\tau,\xi)\partial_{t,x}(\mu c_{jk}\chi)  \right).
    \end{split}
\end{equation}
In view of the uniform bonds on $|\tilde{\chi}^{+,\eta}(\tau,\xi)|^{2}$ and on $|\partial_{t,x}(\mu c_{j,k}\chi)|^{2}$, the contribution to $[\mu c_{jk}\chi,\chi R^\eta]R^\eta$ from the first term is uniformly (with respect to $\eta$) bounded in $L^{2}$. The contribution from the second term is (modulo a constant factor)
\begin{multline}\label{1}
 \mathcal{F}^{-1}\big((\tau-\tilde{M}_\eta)\chipt(\tau,\xi)\partial_{\tau,\xi}\chipt(\tau,\xi)\partial_{t,x}(\mu c_{jk}\chi)\mathcal{F}u\big) \\
 =(T-\tilde{M}_{\eta})\partial_{t,x}(\mu c_{jk}\chi)\mathcal{F}^{-1}\big(\chipt(\tau,\xi)\partial_{\tau,\xi}\chipt(\tau,\xi)\mathcal{F}u\big)\;.
\end{multline}
To estimate the $L^{2}$--norm of the right hand side of \eqref{1}, it suffices to estimate the $L^{2}$--norm of $(-i\frac{\partial}{\partial t}-\tilde{M}_{\eta})\mathcal{F}^{-1}\big(\chipt(\tau,\xi)$ $\partial_{\tau,\xi}\chipt(\tau,\xi)\big)\mathcal{F}u$ ($[T-\tilde{M}_{\eta}, \partial_{t,x}(\mu c_{j,k}\chi)] = [T, \partial_{t,x}(\mu c_{j,k}\chi)]$ is bounded in $L^{2}$, uniformly in $\eta$, since derivatives of $\mu c_{jk}\chi$ are bounded). We first estimate the contribution for the $(-i\frac{\partial}{\partial t})$--term, that is, $\|\mathcal{F}^{-1}\big(\chipt(\tau,\xi)\partial_{\tau,\xi}\chipt(\tau,\xi)\mathcal{F}u\big)\|_{1}$. Observe that 
\begin{equation}\label{2}
\chipt = \tilde{\chi}^{+}\sigma(\tau-3\tilde{M}_{\eta}+1) = \tilde{\chi}^{+} + \tilde{\chi}^{+}\big(\sigma(\tau-3\tilde{M}_{\eta}+1)-1\big)\;,
\end{equation}
and that the second term on the right hand side is compactly supported (but with support depending on $\eta$). It follows that $\tilde{\chi}^{+,\eta}\partial_{\tau,\xi}\tilde{\chi}^{+,\eta} = \tilde{\chi}^{+}\partial_{\tau,\xi}\tilde{\chi}^{+} + \psi_{\eta}$, where $\psi_{\eta}$ is compactly supported (the support depends on $\eta$). Then
\begin{multline}\label{split}
 \|\mathcal{F}^{-1}\big(\chipt\partial_{\tau,\xi}\chipt\mathcal{F}u\big)\|_{1} 
 \leq \|\mathcal{F}^{-1}\big(\tilde{\chi}^{+}\partial_{\tau,\xi}\tilde{\chi}^{+}\mathcal{F}u\|_{1} + \|\mathcal{F}^{-1}\big(\psi_{\eta}\mathcal{F}u\big)\|_{1} \\
 \lesssim \|u\| + C_{\eta}\|u\|_{-1}
 \lesssim \|\overline{\partial}_{b}u\| + \|\overline{\partial}_{b}^{*}u\| + C_{\eta}\|u\|_{-1}\;.
\end{multline}
The second to last inequality follows because $|\tilde{\chi}^{+}\partial_{\tau,\xi}\tilde{\chi}^{+}|\lesssim (1+|(\tau,\xi)|^{2})^{-1/2}$ for the first term, and because $\psi_{\eta}$ is compactly supported for the second. Again using that $\|u\|$ dominates $\|\mathcal{F}^{-1}\big(\tilde{\chi}^{+}\partial_{\tau,\xi}\tilde{\chi}^{+}\mathcal{F}u\|_{1}$ and interpolation gives that the contribution of the $M_{\eta}$--term to $\big\|(-i\frac{\partial}{\partial t}-\tilde{M}_{\eta})\mathcal{F}^{-1}\big(\chipt$ $\partial_{\tau,\xi}\chipt\big)\mathcal{F}u\big\|$ is also dominated by $\|u\| + C_{\eta}\|u\|_{-1}\lesssim \|\overline{\partial}_{b}u\| + \|\overline{\partial}_{b}^{*}u\| + C_{\eta}\|u\|_{-1}$.

Combining the above and using Cauchy--Schwarz lets us estimate the right hand side of \eqref{prop-eq8}:
\begin{multline}\label{prop-eq12}
 \left|\isum_{K}\sum_{j,k} \int_{b\O} [\mu\chi^{2} c_{jk}, R^\eta] R^\eta\F^{-1}\left( \chip \F u_{jK}\right) \overline{ \F^{-1} (\chip \F u_{kK})}\right| \\
 \leq C(\|\overline{\partial}_{b}u\|^{2}+\|\overline{\partial}_{b}^{*}u\|^{2})+C_{\eta}\|u\|_{-1}^{2}\;.
\end{multline}
We have also used that for forms supported in $U^{\prime}$, the norms on $b\Omega$ and those on $\mathbb{R}^{2n-1}$ are equivalent (uniformly in $\eta$) and that the contribution to $[\mu c_{jk}\chi,\chi R^\eta]R^\eta$ due to the portion of the symbol that is in $S^{-1}$ has its $L^{2}$--norm dominated by $C_{\eta}\|\mathcal{F}^{-1}\chi^{+,\eta}\mathcal{F}u_{j,K}\|_{-1}$. In turn, this quantity is dominated by $C_{\eta}\|u\|_{-1}$ ($\chi^{+,\eta}$ is a bounded multiplier).

\smallskip

Combining \eqref{prop-eq4}--\eqref{prop-eq8} and \eqref{prop-eq12} gives
\begin{multline}\label{prop-eq14}
       -(1-\eta)2\text{Re}\left( \isum_{K}\sum_{jk} \int_{b\O} c_{jk} \left( (T - \tilde{M}_\eta ) \pplus u_{jK}\right) \overline{\pplus u_{kK}} \right) \\
       \leq C(1-\eta)(\|\overline{\partial}_{b}u\|^{2}+\|\overline{\partial}_{b}^{*}u\|^{2}) + C_\eta\norm{u}^2_{-1},
\end{multline}
and thus, with \eqref{prop-eq3a},
\begin{multline}\label{prop-eq12ab}
\sideset{}{'}\sum_{J}|\alpha_{\eta}(\tilde{L}^{J}_{\mathcal{P}^{+,\eta}u})|^{2}\\
\lesssim (1-\eta)\left(\|\overline{\partial}_{b}(e^{-h_{\eta}/2}\mathcal{P}^{+,\eta}u)\|_{-h_{\eta}}^{2}+\|\overline{\partial}_{b,-h_{\eta}}^{*}(e^{-h_{\eta}/2}\mathcal{P}^{+,\eta}u)\|_{-h_{\eta}}^{2}+\|\overline{\partial}_{b}u\|^{2}+\|\overline{\partial}_{b}^{*}u\|^{2}\right) \\
+ C_{\eta}\|u\|_{-1}^{2}\;.\;\;\;\;\;\;\;\;
\end{multline}
We will deal with the weighted terms later and for now proceed to $\mathcal{P}^{-,\eta}u$.

\smallskip

In order to deal with $\pminus u$, we first modify \eqref{prop-eq1}. But rather than keeping $\|\overline{L}(e^{-h_{\eta}/2}u)\|^{2}$ and integrating this term by parts (compare \cite{Harrington-Raich-2011}, formula (12)), we make a somewhat ad hoc modification: we first drop the term (since it is nonnegative) and then add $3q\|\overline{L}u\|^{2}$ on both sides (the reason for the factor $3q$ will become clear below). Rearranging the inequality gives
\begin{multline}\label{prop-eq15}
         2\isum_{K}\int_{b\O}\partial\overline{\partial}(-h_{\eta})\big(\tilde{L}_{u}^{K}\wedge\overline{\tilde{L}_{u}^{K}}\big) \\
         + 2\text{Re} \isum_{K}\sum_{j,k}\int_{b\O}  c_{jk}(T+\frac{1}{2}\frac{\partial h_{\eta}}{\partial\nu}) + 3q\norm{\overline{L}u}^2 \\
          \lesssim \normwh{\dbarb (e^{-h_\eta/2}u)}^2 + \normwh{\dbarbstarwh (e^{-h_\eta/2}u)}^2 + \bigo{\norm{u}^2} + 3q\norm{\overline{L}u}^2.
\end{multline}
On the left hand side, we integrate $\|\overline{L}u\|^{2}$ by parts as usual. That is
\begin{equation}\label{}
    \norm{\overline{L}u}^2 = - \sum_j\isum_J \int_{b\O} (c_{jj} Tu_J)\overline{u_{J}} +  \norm{Lu}^2 + \bigo{\norm{{L}u}\norm{u}} + \bigo{\norm{\overline{L}u}\norm{u}}\;;
\end{equation}
this results in
\begin{multline}\label{prop-eq16}
         2\isum_{K}\int_{b\O}\partial\overline{\partial}(-h_{\eta})\big(\tilde{L}_{u}^{K}\wedge\overline{\tilde{L}_{u}^{K}}\big) + 2\text{Re} \isum_{K}\sum_{j,k}\int_{b\O}  c_{jk}(T+\frac{1}{2}\frac{\partial h_{\eta}}{\partial\nu})u_{jK}\overline{u_{kK}} \\
          - 3q\sum_j\isum_J \int_{b\O} (c_{jj} Tu_J)\overline{u_J} +  3q\norm{Lu}^2 + \bigo{\norm{{L}u}\norm{u}} + \bigo{\norm{\overline{L}u}\norm{u}} \\
          \lesssim \normwh{\dbarb (e^{-h_\eta/2}u)}^2 + \normwh{\dbarbstarwh (e^{-h_\eta/2}u)}^2 + \bigo{\norm{u}^2} + 3\norm{\overline{L}u}^2.
\end{multline}
Using $\|\overline{L}u\|^{2}\geq 0$, the Cauchy-Schwarz inequality, maximal estimates, and \eqref{basic} gives
\begin{multline}\label{prop-eq17}
        2\isum_{K}\int_{b\O}\partial\overline{\partial}(-h_{\eta})\big(\tilde{L}_{u}^{K}\wedge\overline{\tilde{L}_{u}^{K}}\big)  \\
         + 2\text{Re} \isum_{K}\sum_{j,k}\int_{b\O}  c_{jk}(T+\frac{1}{2}\frac{\partial h_{\eta}}{\partial\nu}) 
         u_{jK}\overline{u_{kK}} 
          - 3q\sum_j\isum_J \int_{b\O} (c_{jj} Tu_J)\overline{u_J} \\
          \lesssim \normwh{\dbarb (e^{-h_\eta/2}u)}^2 + \normwh{\dbarbstarwh (e^{-h_\eta/2}u)}^2 + \norm{\dbarb u}^2 + \norm{\dbarbstar u}^2.
\end{multline}
Since $T$ is purely imaginary, integration by parts shows that $\sum_j\sum_J' \int_{b\O} (c_{jj} Tu_J)\overline{u_J}$ is real modulo $\bigo{\norm{u}^2}$. Therefore taking real parts in \eqref{prop-eq17}, using \eqref{basic} again, together with $\sum_j\sum^{'}_{|J|=q} \int_{b\O} (c_{jj} Tu_J)\overline{u_J} = \frac{1}{q} \sum_l\sum_{j,k}\sum^{'}_{|K|=q-1} \int_{b\O} (c_{ll}\delta_{jk} Tu_{jK})\overline{u_{kK}}$ (the factor of $1/q$ arises because each $q$-tuple $J$ gives rise to $q$ of the tuples $jK$) results in
\begin{multline}\label{prop-eq20}
    2\isum_{K}\int_{b\O}\partial\overline{\partial}(-h_{\eta})\big(\tilde{L}_{u}^{K}\wedge\overline{\tilde{L}_{u}^{K}}\big)  \\
         + \text{Re} \isum_{K}\sum_{j,k}\int_{b\O} \left( \left (2c_{jk} - 3\delta_{jk}\text{Tr}(c)\right) T +\frac{\partial h_{\eta}}{\partial\nu}c_{jk}\right) u_{jK}\overline{u_{kK}} \\
         \lesssim \normwh{\dbarb (e^{-h_\eta/2}u)}^2 + \normwh{\dbarbstarwh (e^{-h_\eta/2}u)}^2 + \norm{\dbarb u}^2 + \norm{\dbarbstar u}^2\;,
\end{multline}
where $Tr(c)$ denotes the trace of the Levi form. Multiplying both sides of \eqref{prop-eq20} by $(1-\eta)$, then adding $M_\eta \sum^{'}_{K}\sum_{j,k}\int_{b\O} c_{jk}u_{jK}\overline{u_{kK}}$, and invoking \eqref{alphaest2} and \eqref{basic} leads to
\begin{multline}\label{prop-eq21}
         \sideset{}{'}\sum_{J}\int_{b\Omega}|\alpha_{\eta}(\tilde{L}_{u}^{J})|^{2}\\
          \lesssim (1-\eta)\bigg (\normwh{\dbarb (e^{-h_\eta/2}u)}^2 + \normwh{\dbarbstarwh (e^{-h_\eta/2}u)}^2 + \norm{\dbarb u}^2 + \norm{\dbarbstar u}^2 \\
          - \text{Re} \isum_{K}\sum_{j,k}\int_{b\O}  \left (2c_{jk} - 3\delta_{jk}\text{Tr}(c)\right) T u_{jK}\overline{u_{kK}} 
          + \tilde{M}_\eta \isum_{K}\sum_{j,k}\int_{b\O} c_{jk}u_{jK}\overline{u_{kK}} \bigg ),
\end{multline}
where $\tilde{M}_{\eta}$ is chosen big enough so that $\frac{M_{\eta}}{1-\eta}-\frac{\partial h_{\eta}}{\partial\nu}\leq \tilde{M}_{\eta}$. We now need an upper bound on the last two terms on the right hand side of \eqref{prop-eq21}, with $u$ replaced by $\pminus u$, equivalently, a lower bound on
\begin{multline}\label{prop-eq22}
\text{Re}\bigg( \isum_{K}\sum_{j,k}\int_{b\O}  \left (2c_{jk} - 3\delta_{jk}\text{Tr}(c)\right) T (\pminus u)_{jK}\overline{(\pminus u)_{kK}}
        \\  \qquad\qquad\qquad\qquad  - \tilde{M}_\eta \isum_{K}\sum_{j,k}\int_{b\O} c_{jk}(\pminus u)_{jK}\overline{(\pminus u)_{kK}} \bigg )
        \\  = \text{Re}\bigg( \isum_{K}\sum_{j,k}\int_{b\O}  \left (2c_{jk} - 3\delta_{jk}\text{Tr}(c)\right) (T+ \tilde{M}_\eta ) (\pminus u)_{jK}\overline{(\pminus u)_{kK}}
        \\ \qquad\qquad\qquad - 3\tilde{M}_\eta\isum_{K}\sum_{j,k}\int_{b\O}\left( c_{jk} - \delta_{jk}\text{Tr}(c) \right)(\pminus u)_{jK}\overline{(\pminus u)_{kK}}
        \bigg)
    \\ \ge \text{Re} \bigg( \isum_{K}\sum_{j,k}\int_{b\O}  \left (c_{jk} - 2\delta_{jk}\text{Tr}(c)\right) (T+ \tilde{M}_\eta) (\pminus u)_{jK}\overline{(\pminus u)_{kK}}
        \bigg) \\
    = \text{Re} \bigg( \isum_{K}\sum_{j,k}\int_{b\O}  \left (2\delta_{jk}\text{Tr}(c)-c_{jk}\right) (-T-\tilde{M}_\eta) (\pminus u)_{jK}\overline{(\pminus u)_{kK}}
        \bigg) \;.
\end{multline}
The inequality follows because the matrix $(c_{jk}-\delta_{jk}\text{Tr}(c))$ is negative semidefinite. The right hand side of \eqref{prop-eq22} is analogous to \eqref{prop-eq4}, with $c_{jk}$, $T-\tilde{M}_{\eta}$, and $\chi^{+,\eta}$ replaced by $(2\delta_{jk}\text{Tr}(c)-c_{jk})$, $-T-\tilde{M}_{\eta}$, and $\chi^{-,\eta}$, respectively. On the support of $\chi^{-,\eta}$, $-\tau - \tilde{M}_{\eta}\geq 1+2\tilde{M}_{\eta}\geq 1$. With this set up, the right hand side of \eqref{prop-eq22} can now be estimated in the same way as \eqref{prop-eq4} to obtain the analogue of \eqref{prop-eq12ab} for $\mathcal{P}^{-,\eta}u$:
\begin{multline}\label{prop-eq12abc}
\sideset{}{'}\sum_{J}|\alpha_{\eta}(\tilde{L}^{J}_{\mathcal{P}^{-,\eta}u})|^{2}\\
\lesssim (1-\eta)\left(\|\overline{\partial}_{b}(e^{-h_{\eta}/2}\mathcal{P}^{-,\eta}u)\|_{-h_{\eta}}^{2}+\|\overline{\partial}_{b,-h_{\eta}}^{*}(e^{-h_{\eta}/2}\mathcal{P}^{-,\eta}u)\|_{-h_{\eta}}^{2}+\|\overline{\partial}_{b}u\|^{2}+\|\overline{\partial}_{b}^{*}u\|^{2}\right) \\
+ C_{\eta}\|u\|_{-1}^{2}\;.\;\;\;\;\;\;\;\;\;
\end{multline}
$\;\;$To estimate the Re--term in the second line of \eqref{prop-eq3a} for $\mathcal{P}^{0,\eta}u$, we again have to control the dependence of our estimates on $\eta$. To this end, note that $\chi^{0,\eta} = \chi^{0} + \chi^{+}(1-\sigma(\tau-3\tilde{M}_{\eta}))+\chi^{-}(1-\sigma(-\tau-3\tilde{M}_{\eta}))$. $\chi^{0}$ is independent of $\eta$ and is supported in an elliptic region, allowing estimates that are independent of $\eta$. The support of $\chi^{+}(1-\sigma(\tau-3\tilde{M}_{\eta}))+\chi_{-}(1-\sigma(-\tau-3\tilde{M}_{\eta}))$ depends on $\eta$, but is compact, so that the corresponding operator is smoothing of infinite order (this is analogous to \eqref{2} -- \eqref{split}). Accordingly, we have $\|\mathcal{P}^{0,\eta}u\|_{1} \lesssim \|\overline{\partial}_{b}\mathcal{P}^{0}u\|+\|\overline{\partial}_{b}^{*}\mathcal{P}^{0}u\| + C_{\eta}\|u\|_{-1}$. Therefore
\begin{multline}\label{prop-eq-22c}
 \left|\int_{b\O} c_{jk} (T \pzero u_{jK}) \overline{\pzero u_{kK}}\,\right|\lesssim \|\pzero u\|_{1}\|\pzero u\|  \\
 \lesssim \big(\|\overline{\partial}_{b}\mathcal{P}^{0}u\|+\|\overline{\partial}_{b}^{*}u\mathcal{P}^{0}\|+C_{\eta}\|u\|_{-1}\big)\|u\| \\
 \lesssim \|\overline{\partial}_{b}u\|^{2}+\|\overline{\partial}_{b}^{*}u\|^{2}+C_{\eta}\|u\|_{-1}^{2}\;. 
\end{multline}
The last inequality holds because the commutators $[\overline{\partial}_{b},\mathcal{P}^{0}]$ and $[\overline{\partial}_{b}^{*},\mathcal{P}^{0}]$ are order zero and are independent of $\eta$. We also used \eqref{basic}. In addition,
\begin{multline}\label{prop-eq-22e}
\tilde{M}_{\eta}\left|\int_{b\O} c_{jk}\, \pzero u_{jK}\, \overline{\pzero u_{kK}}\,\right| \lesssim \tilde{M}_{\eta}\|\pzero u\|^{2}    \\
\lesssim \tilde{M}_{\eta}\left(\frac{1}{\tilde{M}_{\eta}}\|\pzero u\|_{1}+ C_{\eta}\|\pzero u\|_{-1}\right)\|\mathcal{P}^{0,\eta}u\|  \\
\lesssim \|\overline{\partial}_{b}u\|^{2}+\|\overline{\partial}_{b}^{*}u\|^{2}+C_{\eta}\|u\|_{-1}^{2}\;.
\end{multline}
We have used that $\|\pzero u\|_{-1}\leq C_{\eta}\|u\|_{-1}$ ($\pzero$ is of order zero uniformly in $\eta$).
Combining \eqref{prop-eq-22c} and \eqref{prop-eq-22e} now also establishes the estimate
\begin{multline}\label{prop-eq-22d}
       -(1-\eta)2\text{Re}\left( \isum_{K}\sum_{jk} \int_{b\O} c_{jk} \left( (T - \tilde{M}_\eta ) \pzero u_{jK}\right) \overline{\pzero u_{kK}} \right) \\
       \leq C(1-\eta)(\|\overline{\partial}_{b}u\|^{2}+\|\overline{\partial}_{b}^{*}u\|^{2}) + C_\eta\norm{u}^2_{-1}\;.
\end{multline}
With \eqref{prop-eq3a}, we thus have
\begin{multline}\label{prop-eq-22dd}
 \sideset{}{'}\sum_{J}|\alpha_{\eta}(\tilde{L}^{J}_{\mathcal{P}^{0,\eta}u})|^{2}\\
\lesssim (1-\eta)\left(\|\overline{\partial}_{b}(e^{-h_{\eta}/2}\mathcal{P}^{0,\eta}u)\|_{-h_{\eta}}^{2}+\|\overline{\partial}_{b,-h_{\eta}}^{*}(e^{-h_{\eta}/2}\mathcal{P}^{0,\eta}u)\|_{-h_{\eta}}^{2}+\|\overline{\partial}_{b}u\|^{2}+\|\overline{\partial}_{b}^{*}u\|^{2}\right) \\
+ C_{\eta}\|u\|_{-1}^{2}\;.\;\;\;\;\;\;\;\;\;
\end{multline}

\smallskip

We now deal with the `weighted' terms in the right hand sides of \eqref{prop-eq12ab}, \eqref{prop-eq12abc}, and \eqref{prop-eq-22dd}. The argument is the same in all three cases, we give it for \eqref{prop-eq12ab}. We closely follow \cite{Liu-Straube-2025}. In a special boundary chart, we have
\begin{multline}\label{wt-eq1}
        \dbarstar_{b,-h_\eta}(e^{-h_\eta/2}\mathcal{P}^{+,\eta}u) \\
        = e^{-h_{\eta}}\overline{\partial}_{b}^{*}(e^{h_{\eta}}e^{-h_{\eta}/2}\mathcal{P}^{+,\eta}u) = e^{-h_{\eta}}\overline{\partial}_{b}^{*}(e^{h_{\eta}/2}\mathcal{P}^{+,\eta}u)\;\;\;\;\;\;\;\;\;\;\;\;\;\;\;\;\;\;\;\;\;\;\;\;\;\;\;\;\;\;\\
        \;\;= e^{-h_\eta/2}\dbarbstar \mathcal{P}^{+,\eta}u + \frac{1}{2}\sum_j\isum_K e^{-h_\eta/2} (L_j h_\eta)(\mathcal{P}^{+,\eta}u_{jK})\overline{\o}_K + \bigo{|e^{-h_\eta/2}\mathcal{P}^{+,\eta}u|}\;,
\end{multline}
where the constant in the $\mathcal{O}$--term does not depend on $\eta$.
Taking $\normwh{\cdot}$-norm on both sides of \eqref{wt-eq1}, we get 
\begin{equation}\label{wt-eq2}
    \normwh{\dbarbstarwh(e^{-h_\eta/2}\mathcal{P}^{+,\eta}u)}^2 \lesssim \norm{\dbarbstar \mathcal{P}^{+,\eta}u}^2 + \big\|\sum_j\isum_K(L_j h_\eta)\mathcal{P}^{+,\eta}u_{jK}\big\|^2 + \norm{u}^2.
\end{equation}
Now 
\begin{multline}\label{wt-eq3}
        \Big\|{\sum_j\isum_K (L_j h_\eta)\mathcal{P}^{+,\eta}u_{jK}}\Big\|^2 \\
        = \Big\|{\sum_j\isum_K dh(L_{j})\mathcal{P}^{+,\eta}u_{jK}}\Big\|^2 = \Big\| \sum_j\isum_{K} \left( \alpha_\eta - \alpha^\rho \right)(L_{j})\mathcal{P}^{+,\eta}u_{jK} \Big\|^2 \\
         \lesssim  \isum_K \|\alpha_\eta\big(\sum_{j}(\mathcal{P}^{+,\eta}u)_{jK}L_{j}\big)\|^2 +  \sum_{j}\isum_K\| \alpha^\rho(L_{j})\mathcal{P}^{+,\eta}u_{jK}\|^2 \\
        \lesssim  \isum_K \|\alpha_\eta(\tilde{L}_{\mathcal{P}^{+,\eta}u}^K)\|^2 + \norm{u}^2.
\end{multline}
We have used that $\|\mathcal{P}^{+,\eta}u\|^{2}\leq \|u\|^{2}$. Estimates \eqref{wt-eq2} and \eqref{wt-eq3} combine to give
\begin{equation}\label{wt-eq3a}
    \normwh{\dbarbstarwh(e^{-h_\eta/2}\mathcal{P}^{+,\eta}u)}^2 \lesssim \norm{\dbarbstar \mathcal{P}^{+,\eta}u}^2 + \isum_K \norm{\alpha_\eta(\tilde{L}_{\mathcal{P}^{+,\eta}u}^K)}^2 + \norm{u}^2.
\end{equation}
Inserting \eqref{wt-eq3a} into \eqref{prop-eq12ab} gives
\begin{equation}\label{wt-eq4}
    \begin{split}
        \isum_{K}\int_{b\O} \abs{\alpha_\eta(\tilde{L}_{\mathcal{P}^{+,\eta}u}^K)}^2 &\lesssim (1-\eta)\Big( \normwh{\dbarb(e^{-h_\eta/2}\mathcal{P}^{+,\eta}u)}^2 + \norm{\dbarbstar \mathcal{P}^{+,\eta}u}^2 + \norm{u}^2
        \\ & \qquad + \isum_K\norm{\alpha_\eta(\tilde{L}_{\mathcal{P}^{+,\eta}u}^K)}^2 + \norm{\dbarb u}^2+\|\overline{\partial}_{b}^{*}u\|^{2}\Big) + C_\eta\norm{u}_{-1}^{2}.
    \end{split}
\end{equation}
For $\eta$ sufficiently close to $1$, we absorb $(1-\eta)\sum^{'}_K\norm{\alpha_\eta(\tilde{L}_{\mathcal{P}^{+,\eta}u}^K)}^2=(1-\eta)\sum_{K}'\int_{b\O} \abs{\alpha_\eta(\tilde{L}_{\mathcal{P}^{+,\eta}u}^K)}^2$ into the left hand side of \eqref{wt-eq4} and use $\norm{u}^2\lesssim \norm{\dbarb u}^2 + \norm{\dbarbstar u}^2$ to obtain
\begin{multline}\label{wt-eq5}
        \isum_K \int_{b\O} \abs{\alpha_\eta(\tilde{L}_{\mathcal{P}^{+,\eta}u}^K)}^2 \lesssim (1-\eta)\Big( \normwh{\dbarb(e^{-h_\eta/2}\mathcal{P}^{+,\eta}u)}^2 +\|\overline{\partial}_{b}^{*}\mathcal{P}^{+,\eta}u\|^{2} \\
       + \norm{\dbarb u}^2 + \norm{\dbarbstar u}^2 \Big) +  C_\eta\norm{u}_{-1}^{2}.
\end{multline}

\smallskip

Now we estimate $\normwh{\dbarb(e^{-h_\eta/2}\mathcal{P}^{+,\eta}u)}^2$. We have
\begin{multline}\label{wt-eq6}
   \| \dbarb(e^{-h_\eta/2}\mathcal{P}^{+,\eta}u)\|_{-h_{\eta}}^{2} \\
   \lesssim \|e^{-h_\eta/2}\dbarb \mathcal{P}^{+,\eta}u\|_{-h_{\eta}}^{2} + \|e^{-h_\eta/2}(\dbarb h_\eta)\wedge \mathcal{P}^{+,\eta}u\|_{-h_{\eta}}^{2} \\
\lesssim \norm{\dbarb \mathcal{P}^{+,\eta}u}^2 + \norm{(\dbarb h_\eta)\wedge \mathcal{P}^{+,\eta}u}^2.
\end{multline}
For the last term, note that because $h_{\eta}$ is real, $|\overline{L_{j}}h_{\eta}|=|L_{j}h_{\eta}|$, so that
\begin{multline}\label{wt-eq7}
    \norm{\dbarb h_\eta\wedge \mathcal{P}^{+,\eta}u}^2 \le \sum_j\isum_J \int_{b\O}  \abs{(\overline{L}_jh_\eta) \mathcal{P}^{+,\eta}u_J}^2 \\
    = \sum_j\isum_J \int_{b\O}  \abs{({L}_jh_\eta) \mathcal{P}^{+,\eta}u_J}^2
        \lesssim \sum_j\isum_J \int_{b\O}  \abs{\alpha_\eta(\mathcal{P}^{+,\eta}u_JL_j)}^2 + \norm{\mathcal{P}^{+,\eta}u}^2.
\end{multline}
The last inequality follows as in \eqref{wt-eq3}. For each pair $j,J$, let $J'\subset J$ be a $(q-1)$-tuple that does not contain $j$. Then $\mathcal{P}^{+,\eta}(u_{J}\overline{\omega_{j}}\wedge\overline{\omega_{J^{\prime}}})=(\mathcal{P}^{+,\eta}u_{J})\overline{\omega_{j}}\wedge\overline{\omega_{J^{\prime}}}$, and $\sum_{s}(\mathcal{P}^{+,\eta}u_{J}\overline{w_{j}}\wedge\overline{\omega_{J^{\prime}}})_{sJ^{\prime}}L_{s}=\mathcal{P}^{+,\eta}u_{J}L_{j}$ (the sum is only over $s\notin J^{\prime}$, which forces $s=j$ for a nonvanishing term). That is, $\mathcal{P}^{+,\eta}u_{J}L_{j}=\tilde{L}^{J^{\prime}}_{\mathcal{P}^{+,\eta}(u_{J}\overline{\omega_{j}}\wedge\overline{\omega_{J^{\prime}}})}$. (When $q=1$, $J^{\prime}$ is empty, and $(\mathcal{P}^{+,\eta}u_{J})L_{j}=\tilde{L}_{\mathcal{P}^{+,\eta}}u_{J}\overline{\omega_{j}}$.) We now apply \eqref{wt-eq5} with $u$ replaced by $u_{J}\overline{\omega_{j}}\wedge\overline{\omega_{J}}$. Then
\begin{multline}\label{wt-eq7a}
    \int_{b\O}  \abs{\alpha_\eta(\mathcal{P}^{+,\eta}u_JL_j)}^2 
    = \int_{b\O} \abs{\alpha_\eta\big(\tilde{L}^{J^{\prime}}_{\mathcal{P}^{+,\eta}(u_{J}\overline{\omega_{j}}\wedge\overline{\omega_{J^{\prime}}})}\big)}^2
        \leq \sum_{K}\int_{b\O} \abs{\alpha_\eta\big(\tilde{L}^{K}_{\mathcal{P}^{+,\eta}(u_{J}\overline{\omega_{j}}\wedge\overline{\omega_{J^{\prime}}})}\big)}^2 \\
         \lesssim (1-\eta)\Big( \normwh{\dbarb(e^{-h_\eta/2}\mathcal{P}^{+,\eta}u_J \bar{\o}_j\wedge\bar{\o}_{J'})}^2 + \norm{\dbarbstar (\mathcal{P}^{+,\eta}u_J \bar{\o}_j\wedge\bar{\o}_{J'})}^2 \\
         + \norm{\dbarb (u_J \bar{\o}_j\wedge\bar{\o}_{J'})}^2
         + \norm{\dbarbstar (u_J \bar{\o}_j\wedge\bar{\o}_{J'})}^2 \Big) +  C_\eta\norm{u_J \bar{\o}_j\wedge\bar{\o}_{J'}}_{-1}^{2}.
\end{multline}
The last inequality follows from \eqref{wt-eq5}. Next, we estimate the first term on the right hand side of \eqref{wt-eq7a}.
\begin{multline}\label{wt-eq7b}
        \normwh{\dbarb(e^{-h_\eta/2}\mathcal{P}^{+,\eta}u_J \bar{\o}_j\wedge\bar{\o}_{J'})}^2 
        \lesssim \sum_s \normwh{\overline{L}_s(e^{-h_\eta/2}\mathcal{P}^{+,\eta}u_J)}^2 + \normwh{e^{-h_\eta/2}\mathcal{P}^{+,\eta}u_J}^2 \\
        \lesssim \sum_s \left( \norm{(\overline{L}_s h_\eta)\mathcal{P}^{+,\eta}u_J}^2 + \norm{\overline{L}_s \mathcal{P}^{+,\eta}u_J}^2 \right) + \norm{\mathcal{P}^{+,\eta}u_J}^2 \\
        \lesssim \sum_s \norm{(\overline{L}_s h_\eta)\mathcal{P}^{+,\eta}u_J}^2 + \norm{\dbarb \mathcal{P}^{+,\eta}u}^2 + \norm{\dbarbstar \mathcal{P}^{+,\eta}u}^2\;.
\end{multline}
For the third inequality, we have used maximal estimates and \eqref{basic}. Starting with the second term on the first line of \eqref{wt-eq7} and  substituting \eqref{wt-eq7a} into its right hand side, we get
\begin{multline}\label{wt-eq-7bb}
  \sum_j\isum_J \int_{b\O}  \abs{(\overline{L}_jh_\eta)\mathcal{P}^{+,\eta} u_J}^2 \\
   \lesssim (1-\eta)\sum_{j}\sideset{}{'}\sum_{J}\Big( \normwh{\dbarb(e^{-h_\eta/2}\mathcal{P}^{+,\eta}u_J \bar{\o}_j\wedge\bar{\o}_{J'})}^2
   + \norm{\dbarbstar (\mathcal{P}^{+,\eta}u_J \bar{\o}_j\wedge\bar{\o}_{J'})}^2 \\
         + \norm{\dbarb (u_J \bar{\o}_j\wedge\bar{\o}_{J'})}^2
         + \norm{\dbarbstar (u_J \bar{\o}_j\wedge\bar{\o}_{J'})}^2 \Big) \\
         +  C_\eta\sum_{j}\sideset{}{'}\sum_{J}\norm{u_J \bar{\o}_j\wedge\bar{\o}_{J'}}_{-1}^{2} + \|\mathcal{P}^{+,\eta}u\|^{2}\;.
\end{multline}
Now we substitute \eqref{wt-eq7b} for the $\|\cdots\|_{-h_{\eta}}^{2}$--term in the second line of \eqref{wt-eq-7bb} and obtain
\begin{multline}\label{wt-eq7c}
    \sum_j\isum_J \int_{b\O}  \abs{(\overline{L}_jh_\eta)\mathcal{P}^{+,\eta} u_J}^2 \\
  \lesssim (1-\eta)\sum_{j}\sideset{}{'}\sum_{J}\Big( \sum_s \norm{(\overline{L}_s h_\eta)\mathcal{P}^{+,\eta}u_J}^2 + \norm{\dbarb \mathcal{P}^{+,\eta}u}^2 + \norm{\dbarbstar \mathcal{P}^{+,\eta}u}^2  \\
   \;\;\;\;\;\;\;\;\;\;\;\;\;\;\;\;\;\;\;\;\;\;\;+ \norm{\dbarbstar (\mathcal{P}^{+,\eta}u_J \bar{\o}_j\wedge\bar{\o}_{J'})}^2      + \norm{\dbarb (u_J \bar{\o}_j\wedge\bar{\o}_{J'})}^2
         + \norm{\dbarbstar (u_J \bar{\o}_j\wedge\bar{\o}_{J'})}^2 \Big) \\
         +  C_\eta\sum_{j}\sideset{}{'}\sum_{J}\norm{u_J \bar{\o}_j\wedge\bar{\o}_{J'}}_{-1}^{2} + \|\mathcal{P}^{+,\eta}u\|^{2}\;.
\end{multline}
Choosing $(1-\eta)$ small enough and absorbing $(1-\eta)(n-1)\sum_s\sum^{'}_{J} \norm{(\overline{L}_s h_\eta)u_J}^2$ into the left hand side gives
\begin{multline}\label{wt-eq7d}
\sum_j\sideset{}{'}\sum_{J}\norm{(\overline{L}_j h_\eta)\mathcal{P}^{+,\eta}u_J}^2 \\
\lesssim (1-\eta)\Big(\norm{\dbarb \mathcal{P}^{+,\eta}u}^2 + \norm{\dbarbstar \mathcal{P}^{+,\eta}u}^2+ \sum_{j}\sideset{}{'}\sum_{J}\Big(\norm{\dbarbstar (\mathcal{P}^{+,\eta}u_J \bar{\o}_j\wedge\bar{\o}_{J'})}^2  \\
   \;\;\;\;\;\;\;\;\;+ \norm{\dbarb (u_J \bar{\o}_j\wedge\bar{\o}_{J'})}^2
         + \norm{\dbarbstar (u_J \bar{\o}_j\wedge\bar{\o}_{J'})}^2 \Big)\Big) +\|u\|^{2}+C_{\eta}\|u\|_{-1}^{2}\;.
\end{multline}
We have used that $\|\mathcal{P}^{+,\eta}u\|^{2}\lesssim \|u\|^{2}$ and $\norm{u_J \bar{\o}_j\wedge\bar{\o}_{J'}}_{-1}^{2}\lesssim\|u\|_{-1}^{2}$. Starting with the first inequality in \eqref{wt-eq7} and using \eqref{wt-eq7d} gives us
\begin{multline}\label{wt-eq8}
        \norm{\dbarb h_\eta\wedge \mathcal{P}^{+,\eta}u}^2  \\
        \lesssim (1-\eta)\Big(\norm{\dbarb \mathcal{P}^{+,\eta}u}^2 + \norm{\dbarbstar \mathcal{P}^{+,\eta}u}^2+ \sum_{j}\sideset{}{'}\sum_{J}\big(\norm{\dbarbstar (\mathcal{P}^{+,\eta}u_J \bar{\o}_j\wedge\bar{\o}_{J'})}^2  \\
   \;\;\;\;\;\;\;\;\;+ \norm{\dbarb (u_J \bar{\o}_j\wedge\bar{\o}_{J'})}^2
         + \norm{\dbarbstar (u_J \bar{\o}_j\wedge\bar{\o}_{J'})}^2\big) \Big) +\|u\|^{2}+C_{\eta}\|u\|_{-1}^{2}\;.
\end{multline}
Combining \eqref{wt-eq5}, \eqref{wt-eq6}, and \eqref{wt-eq8}, and using $(1-\eta)^2 \le (1-\eta)$ and \eqref{basic}, we get
\begin{multline}\label{wt-eq9}
        \isum_K \int_{b\O} \abs{\alpha_\eta(\tilde{L}_{\mathcal{P}^{+,\eta}u}^K)}^2 \\
    \lesssim (1-\eta)\Big(\norm{\dbarb \mathcal{P}^{+,\eta}u}^2 + \norm{\dbarbstar \mathcal{P}^{+,\eta}u}^2+ \sum_{j}\sideset{}{'}\sum_{J}\Big(\norm{\dbarbstar (\mathcal{P}^{+,\eta}u_J \bar{\o}_j\wedge\bar{\o}_{J'})}^2  \\
   \;\;\;\;\;\;\;\;\;\;\;\;\;\;\;\;\;\;\;\;\;\;\;\;\;\;\;\;\;\;+ \norm{\dbarb (u_J \bar{\o}_j\wedge\bar{\o}_{J'})}^2
         + \norm{\dbarbstar (u_J \bar{\o}_j\wedge\bar{\o}_{J'})}^2\Big) +\|\overline{\partial}_{b}u\|^{2}+\|\overline{\partial}_{b}^{*}u\|^{2}\Big)
         \\
         +C_{\eta}\|u\|_{-1}^{2}\;.
\end{multline}
The various terms on the right hand side of \eqref{wt-eq9} are estimated as follows. First,
\begin{multline}\label{wt-eq10}
        \norm{\dbarb (u_J \bar{\o}_j\wedge\bar{\o}_{J'})}^2 + \norm{\dbarbstar (u_J \bar{\o}_j\wedge\bar{\o}_{J'})}^2 \\
        \lesssim \sum_s \left(
    \norm{\overline{L}_s u_J}^2 + \norm{{L}_s u_J}^2 \right) + \norm{u}^2
         \lesssim \norm{\dbarb u}^2 + \norm{\dbarbstar u}^2.
\end{multline}
We have again used maximal estimates and \eqref{basic}. In the terms in the first line of the right hand side of \eqref{wt-eq9} we commute $\mathcal{P}^{+,\eta}$ with the respective differential operators. Because $(\mathcal{P}^{+,\eta}-\mathcal{P}^{+})$ is infinitely smoothing ($\chi^{+,\eta}$ and $\chi^{+}$ differ by a compactly supported symbol), and $[\overline{\partial}_{b},\mathcal{P}^{+}]$ is of order zero (and does not depend on $\eta$), we can prove estimates that are uniform in $\eta$ (modulo a weak term):
\begin{multline}\label{wt-eq10a}
 [\overline{\partial}_{b}\mathcal{P}^{+,\eta}u\|^{2}\lesssim \|\overline{\partial}_{b}\mathcal{P}^{+}u\|^{2}+\|\overline{\partial}_{b}(\mathcal{P}^{+,\eta}-\mathcal{P}^{+})u\|^{2} \\
 \lesssim \|\mathcal{P}^{+}\overline{\partial}_{b}u\|^{2}+\|[\overline{\partial}_{b},\mathcal{P}^{+}]u\|^{2}+\|(\mathcal{P}^{+,\eta}-\mathcal{P}^{+})u\|_{1}^{2} \\
 \lesssim \|\overline{\partial}_{b}u\|^{2}+\|u\|^{2} +C_{\eta}\|u\|_{-1}^{2} 
 \lesssim \|\overline{\partial}_{b}u\|^{2}+\|\overline{\partial}_{b}^{*}u\|^{2} +C_{\eta}\|u\|_{-1}^{2}.
\end{multline}
$\|\overline{\partial}_{b}^{*}\mathcal{P}^{+,\eta}u\|^{2}$ and $ \norm{\dbarbstar (\mathcal{P}^{+,\eta}u_J \bar{\o}_j\wedge\bar{\o}_{J'})}^2$ are likewise estimated by the right hand side of \eqref{wt-eq10a} (for the latter, use also \eqref{wt-eq10}). Thus we finally arrive at
\begin{equation}\label{wt-eq11}
    \begin{split}
        \isum_K \int_{b\O} \abs{\alpha_\eta(\tilde{L}_{\mathcal{P}^{+,\eta}u}^K)}^2 &\lesssim (1-\eta)\Big(  \norm{\dbarb u}^2 + \norm{\dbarbstar u}^2 \Big) +  C_\eta\norm{u}_{-1}.
    \end{split}
\end{equation}

\smallskip

The estimates for the `weighted' terms on the right hand sides of \eqref{prop-eq12abc} and \eqref{prop-eq-22dd}, respectively, are verbatim the same, and the proof of Proposition \ref{main-proposition} is complete.
\end{proof}

\section{Proof of Theorem \ref{main-result}}\label{proof}

The prove Theorem \ref{main-result}, we only have to prove the estimates for $s=k\in\mathbb{N}$ (by interpolation). It is shown in \cite{BoasStraube91b} that the estimates $(3)$ for the complex Green operators are an easy consequence of estimates $(1)$ and $(2)$. Namely
\begin{equation}\label{5.36}
 G_{q} = Q_{q}R_{q-1}S_{q} + R_{q}Q_{q+1}(I-S_{q})\;,
\end{equation}
where $R_{q}$ and $Q_{q}$ are the canonical solution operators (at the appropriate form levels) for $\overline{\partial}_{b}$ and $\overline{\partial}_{b}^{*}$, respectively. In view of estimates (1), these operators are continuous in $W^{k}(b\Omega)$. Consequently, \eqref{5.36} displays $G_{q}$ as a composition of operators all of which are continuous in $W^{k}(b\Omega)$.

\smallskip

The proof of the estimates $(1)$ and $(2)$ relies on Proposition \ref{main-proposition}. For $\eta_{0}<\eta<1$, denote by $\rho_{\eta}$ a defining function for $\Omega$ so that $-(-\rho_{\eta})^{\eta}$ is plurisubharmonic (near $b\Omega$). As usual, our arguments will involve absorbing terms. This means that we need to know \emph{a priori} that the norms to be estimated are finite. For this reason, we will work on the boundaries $b\Omega_{\eta,\delta}$ of the subdomains $\Omega_{\eta,\delta}=\{z\in\Omega\,|\,\rho_{\eta}<-\delta\}$, $0<\delta<\delta_{\eta}$. These boundaries are level sets of $\rho_{\eta}$ and so are also level sets of $-(-\rho_{\eta})^{\eta}$, so that $\Omega_{\eta,\delta}$ admits a defining function that is plurisubharmonic at the boundary (while $\Omega$ need not). As a result, all the estimates in Theorem \ref{main-result} hold on $b\Omega_{\eta,\delta}$, by \cite{BoasStraube91b}, and the relevant norms will be finite. When $\Omega$ itself does not admit a plurisubharmonic defining function, the arguments in \cite{BoasStraube91b} do not give control over constants as $\delta\rightarrow 0$. However, we will show below that Proposition \ref{main-proposition} nevertheless implies that constants may be taken uniformly over approximating subdomains. More precisely, we will show that there is $\eta_{2}=\eta_{2}(k)$, $0<\eta_{2}<1$, such that for $\eta_{2}<\eta<1$, there is $\delta_{\eta}=\delta_{\eta}(k)>0$ so that the estimates in Theorem \ref{main-result} (with $s=k$) hold on $b\Omega_{\eta,\delta}$ uniformly for $0<\delta<\delta_{\eta}$. Once we have these uniform estimates, they will transfer to $b\Omega$ as in section 4 of \cite{BoasStraube91b}.

\smallskip

We first prove $(1)$ for $\min\{ q_0,n-1-q_0 \} \le q \le \max\{ q_0,n-1-q_0 \}$. Note that $T_{\eta}$ and $\alpha_{\eta}$ live in a neighborhood of $b\Omega=b\Omega_{\eta}$ and on $b\Omega_{\eta,\delta}$ agree with $T_{(\rho_{\eta}+\delta)}$ and $\alpha_{(\rho_{\eta}+\delta)}$, respectively. Constants used in the proof may vary form one occurrence to the next, but are independent of $\eta$, unless they come with a subscript $\eta$. (Since $k$ is fixed throughout the argument, we do not indicate the dependence on $k$.) The estimates below should be understood as `there is $\delta_{\eta}$ so that the estimate holds uniformly in $\delta$ for $\delta<\delta_{\eta}$'. In order not to clutter the notation, we do not use subscripts $(\eta,\delta)$ on the norms.

\smallskip

Let $\eta>\eta_{0}$. As in \cite{BoasStraube91b}, page 1578, it follows from \eqref{comptang} that
\begin{equation}\label{5.1}
    \norm{u}^2_k \leq C_\eta\left(\norm{T_{\eta}^ku}^2 + \norm{\dbarb u}^2_{k-1} + \norm{\dbarbstar u}^2_{k-1} + \norm{u}_{k-1}\norm{u}_k \right)\;.
\end{equation}
The last term in \eqref{5.1} can be estimated via the usual $s.c.-l.c.$ estimate and interpolation of Sobolev norms as follows:
\begin{equation}\label{5.2}
 \|u\|_{k-1}\|u\|_{k}\leq (s.c.)_{1}\|u\|_{k}^{2}+(l.c.)_{1}\|u\|_{k-1}^{2}\leq (s.c.)_{1}\|u\|_{k}^{2}+(s.c.)_{2}\|u\|_{k}^{2}+(l.c.)_{2}\|u\|^{2}\,,
\end{equation}
where the two small constants can be chosen as small as we wish. Choosing them so that $C_{\eta}((s.c.)_{1}+(s.c.)_{2})\leq 1/2$ lets us absorb the $\|u\|_{k}^{2}$--terms into the left hand side. Taking into account \eqref{basic} then yields 
\begin{equation}\label{5.3}
  \norm{u}^2_k \leq C_\eta\left(\norm{T_{\eta}^ku}^2 + \norm{\dbarb u}^2_{k-1} + \norm{\dbarbstar u}^2_{k-1}\right)\;.
\end{equation}

\smallskip

To estimate $\|T_{\eta}^{k}u\|^{2}$ we will need the commutator formulas \eqref{commutator-formula-dbarb} and \eqref{commutator-formula-dbarbstar}, and so have to work in special boundary charts. Denote by $\{\chi_{s}\}_{s=1}^{M}$ a partition of unity subordinate to a cover $\{U_{s}\}_{s=1}^{M}$ of a neighborhood of $b\Omega$ by special boundary charts (which also serve as charts for $b\Omega_{\eta,\delta}$ for $\delta<\delta_{\eta}$). Note that for a form $u$ supported in a special boundary chart, the difference between $T_{\eta}$ acting componentwise on the components of $u$ in this chart, and $T_{\eta}$ acting componentwise on the Euclidean components, respectively, is of order zero (the difference arises when $T_{\eta}$ acts on the transformation matrix). Thus
\begin{equation}\label{5.4}
 \|T_{\eta}^{k}u\|^{2} \lesssim \big\|\sum_{s=1}^{M}T_{\eta}^{k}(\chi_{s}u)\big\|^{2} + C_{\eta}\|u\|_{k-1}^{2}\;,
\end{equation}
where on the right hand side, $T_{\eta}$ acts coefficientwise on the coefficients in the chart associated with $\chi_{s}$. We now work on the first term on the right hand side of \eqref{5.4}; all computations are in the special boundary charts. In a slight abuse of notation, we refrain from adding an additional sub/super script $s$ to $T_{\eta}$. We have
\begin{multline}\label{5.5}
 \big\|\sum_{s}T_{\eta}^{k}(\chi_{s}u)\big\|^{2}\leq \big\|\sum_{s}\chi_{s}T_{\eta}^{k}u\big\|^{2}+C_{\eta}\|u\|_{k-1}^{2}\\
 \lesssim\big\|\overline{\partial}_{b}\sum_{s}\chi_{s}T_{\eta}^{k}u\big\|^{2}+\big\|\overline{\partial}_{b}^{*}\sum_{s}\chi_{s}T_{\eta}^{k}u\big\|^{2}+c_{\eta}\|u\|_{k-1}^{2}\;.
\end{multline}
The first sum on the right hand is dominated by
\begin{equation}\label{5.6}
 \big\|\sum_{s}(\overline{\partial}_{b}\chi_{s})\wedge T_{\eta}^{k}u\big\|^{2} + \big\|\sum_{s}\chi_{s}\overline{\partial}_{b}(T_{\eta}^{k}u)\big\|^{2}\;.
\end{equation}
The $T_{\eta}^{k}u$ in the first term depends on $s$ ($T_{\eta}$ acts in the chart associated with $\chi_{s}$), but because it differs form the global term $T_{\eta}^{k}u$ computed in Euclidean coordinates by a term of order $(k-1)$, we can still exploit the fact that $\sum_{s}\overline{\partial}_{b}\chi_{s} = 0$ to obtain
\begin{equation}\label{5.7}
\big\|\overline{\partial}_{b}\sum_{s}\chi_{s}T_{\eta}^{k}u\big\|^{2} \leq \big\|\sum_{s}\chi_{s}\overline{\partial}_{b}(T_{\eta}^{k}u)\big\|^{2}+C_{\eta}\|u\|_{k-1}^{2}\;.
\end{equation}
To estimate $\big\|\overline{\partial}_{b}^{*}\sum_{s}\chi_{s}T_{\eta}^{k}u\big\|^{2}$, consider a test form $\gamma$. We have
\begin{equation}\label{5.7a}
 \big(\gamma,\overline{\partial}_{b}^{*}\sum_{s}\chi_{s}T_{\eta}^{k}u\big)=\sum_{s}(\overline{\partial}_{b}\gamma,\chi_{s}T_{\eta}^{k}u) = \sum_{s}(\overline{\partial}_{b}(\chi_{s}\gamma),T_{\eta}^{k}u)-\sum_{s}\big(\overline{\partial}_{b}\chi_{s}\wedge\gamma,T_{\eta}^{k}u\big)\;.
\end{equation}
The first term on the right hand side equals $\big(\gamma,\sum_{s}\chi_{s}\overline{\partial}_{b}^{*}T_{\eta}^{k}u\big)$. In the second term, using again that $T_{\eta}^{k}u$ differs from a globally defined term by an operator of order $(k-1)$ and that $\sum_{s}\overline{\partial}_{b}\chi_{s}=0$ shows that this term is bounded by $C_{\eta}\|\gamma\|\|u\|_{k-1}$. Altogether, we obtain
$\|\overline{\partial}_{b}^{*}\sum_{s}\chi_{s}T_{\eta}^{k}u\|^{2}\lesssim \|\sum_{s}\chi_{s}\overline{\partial}_{b}^{*}T_{\eta}^{k}u\|^{2}+C_{\eta}\|u\|_{k-1}^{2}$. Inserting this estimate and \eqref{5.7} into \eqref{5.5} gives
\begin{multline}\label{5.8}
 \big\|\sum_{s}T_{\eta}^{k}(\chi_{s}u)\big\|^{2} \lesssim \big\|\sum_{s}\chi_{s}\overline{\partial}_{b}(T_{\eta}^{k}u)\big\|^{2}+\big\|\sum_{s}\chi_{s}\overline{\partial}_{b}^{*}(T_{\eta}^{k}u)\big\|^{2}+C_{\eta}\|u\|_{k-1}^{2} \\
 \lesssim\sum_{s}\big\|\chi_{s}\overline{\partial}_{b}(T_{\eta}^{k}u)\big\|^{2}+\sum_{s}\big\|\chi_{s}\overline{\partial}_{b}^{*}(T_{\eta}^{k}u)\big\|^{2}+C_{\eta}\|u\|_{k-1}^{2}\;.
\end{multline}
Now fix $s$. Commuting $T_{\eta}^{k}$ with $\overline{\partial}_{b}$ and $\overline{\partial}_{b}^{*}$, respectively, gives
\begin{multline}\label{5.9}
 \big\|\chi_{s}\overline{\partial}_{b}(T_{\eta}^{k}u)\big\|^{2}+\big\|\chi_{s}\overline{\partial}_{b}^{*}(T_{\eta}^{k}u)\big\|^{2} \\
 \lesssim \big\|\chi_{s}[\overline{\partial}_{b},T_{\eta}^{k}]u\big\|^{2}+\big\|\chi_{s}[\overline{\partial}_{b}^{*},T_{\eta}^{k}]u\big\|^{2}+\|\chi_{s}T_{\eta}^{k}\overline{\partial}_{b}u\|^{2}+\|\chi_{s}T_{\eta}^{k}\overline{\partial}_{b}^{*}u\|^{2}\;.
\end{multline}

\smallskip

Recall from section \ref{ICDA} the vector fields $L_{u}^{J}$, $|J|=q$, associated to a $(0,q)$--form, and the fact that although these vector fields are not invariant under a change form the global Euclidean frame to a special boundary frame, the quantity on the left hand side of \eqref{estimate-main-proposition-euclidean} in Proposition \ref{main-proposition} is, up to a constant factor. Further recall from \eqref{wt-eq7}--\eqref{wt-eq7a} that if $J^{\prime}\subset J$ is a $(q-1)$--tuple with $j\notin J^{\prime}$, then $\tilde{L}_{u_{J}\overline{\omega_{j}}\wedge\overline{\omega_{J^{\prime}}}}^{J^{\prime}}=u_{J}\overline{\omega_{j}}\wedge\overline{\omega_{J^{\prime}}}$. Finally, $|\alpha_{\eta}(\overline{L_{j}})|=|\alpha_{\eta}(L_{j})|$, because $\alpha_{\eta}$ is real. With these facts in mind, the commutator formula \eqref{commutator-formula-dbarb} gives
\begin{multline}\label{5.10}
 \big\|\chi_{s}[\overline{\partial}_{b},T_{\eta}^{k}]u\big\|^{2}\lesssim \sum_{j\notin J}\sideset{}{'}\sum_{J}\|\alpha_{\eta}(\chi_{s}T_{\eta}^{k}u_{J}L_{j})\|^{2}+C_{\eta}\big(\|\overline{\partial}_{b}u\|_{k-1}^{2}+\|\overline{\partial}_{b}^{*}u\|_{k-1}^{2}+\|u\|_{k-1}^{2}\big) \\
 = \sum_{j\notin J}\sum_{J}\int_{b\Omega_{\eta,\delta}}\big|\alpha_{\eta}(\tilde{L}^{J^{\prime}}_{\chi_{s}T_{\eta}^{k}u_{J}(\overline{\omega_{j}}\wedge\overline{\omega_{J^{\prime}}})})\big|^{2}+C_{\eta}\big(\|\overline{\partial}_{b}u\|_{k-1}^{2}+\|\overline{\partial}_{b}^{*}u\|_{k-1}^{2}+\|u\|_{k-1}^{2}\big) \;.
\end{multline}
The sums are only over $j\notin J$ because if $j\in J$, the term $\overline{\omega_{j}}\wedge\overline{\omega_{J}}$ in \eqref{commutator-formula-dbarb} vanishes. For a given $J$, $J^{\prime}\subset J$ is an arbitrary $(q-1)$--tuple (then automatically, $j\notin J^{\prime}$). The $C_{\eta}$-- terms arise from the $A_{h_{\eta}}$--term in \eqref{commutator-formula-dbarb}, upon using maximal estimates for the complex tangential derivative (commuting it, if necessary, so that it acts first) and the $B_{h_{\eta}}$--term. Note that the integral for the fixed $J^{\prime}$ is at most equal to the sum of these integrals over $|K|=(q-1)$, with $K$ replacing the superscript $J^{\prime}$ (but not the subscript). To that sum,we apply Proposition \ref{main-result}. That is 
\begin{multline}\label{5.11}
 \sum_{j\notin J}\sideset{}{'}\sum_{J}\int_{b\Omega_{\eta,\delta}}\big|\alpha_{\eta}(\tilde{L}^{J^{\prime}}_{\chi_{s}T_{\eta}^{k}u_{J}(\overline{\omega_{j}}\wedge\overline{\omega_{J^{\prime}}})})\big|^{2} \leq \sum_{j\notin J}\sideset{}{'}\sum_{J}\sideset{}{'}\sum_{|K|=q-1}\int_{b\Omega_{\eta,\delta}}\big|\alpha_{\eta}(\tilde{L}^{K}_{\chi_{s}T_{\eta}^{k}u_{J}(\overline{\omega_{j}}\wedge\overline{\omega_{J^{\prime}}})})\big|^{2} \\
 \lesssim (1-\eta)\sum_{j\notin J}\sideset{}{'}\sum_{J}\bigg(\|\overline{\partial}_{b}(\chi_{s}T_{\eta}^{k}u_{J}(\overline{\omega_{j}}\wedge\overline{\omega_{J^{\prime}}}))\|^{2}+\|\overline{\partial}_{b}^{*}(\chi_{s}T_{\eta}^{k}u_{J}(\overline{\omega_{j}}\wedge\overline{\omega_{J^{\prime}}}))\|^{2}\bigg) \\
 + C_{\eta}\sum_{j\notin J}\sideset{}{'}\sum_{J}\|\chi_{s}T_{\eta}^{k}u(\overline{\omega_{j}}\wedge\overline{\omega_{J^{\prime}}})\|_{-1}^{2} 
 \end{multline}
Computing $\overline{\partial}_{b}$ and $\overline{\partial}_{b}^{*}$ in the special boundary charts involves only complex tangential derivatives plus undifferentiated terms. To the complex tangential derivatives we apply once more the maximal estimates \eqref{maxest}, for the forms $\chi_{s}T_{\eta}^{k}u$ (coefficients of these forms are being differentiated). We arrive at
 \begin{multline}\label{5.12}
 \sum_{j\notin J}\sideset{}{'}\sum_{J}\int_{b\Omega_{\eta,\delta}}\big|\alpha_{\eta}(\tilde{L}^{J^{\prime}}_{\chi_{s}T_{\eta}^{k}u_{J}(\overline{\omega_{j}}\wedge\overline{\omega_{J^{\prime}}})})\big|^{2} \\
 \lesssim (1-\eta)\big(\|\overline{\partial}_{b}(\chi_{s}T_{\eta}^{k}u)\|^{2}+\|\overline{\partial}_{b}^{*}(\chi_{s}T_{\eta}^{k}u)\|^{2}+\|\chi_{s}T_{\eta}^{k}u\|^{2}\big)+C_{\eta}\|u\|_{k-1}^{2}  \\
 \lesssim (1-\eta)\big(\|\chi_{s}\overline{\partial}_{b}(T_{\eta}^{k}u)\|^{2}+\|\chi_{s}\overline{\partial}_{b}^{*}(T_{\eta}^{k}u)\|^{2}+\|\hat{\chi}_{s}T_{\eta}^{k}u\|^{2}\big) + C_{\eta}\|u\|_{k-1}^{2} \;,
 \end{multline}
 where $\hat{\chi}_{s}$, $s=1\dots M$, are smooth compactly supported in the chart associated with $\chi_{s}$, nonnegative, and identically one in a neighborhood of $\supp(\chi_{s})$. The last inequality follows because the commutators of $\chi_{s}$ with $\overline{\partial}_{b}$ and $\overline{\partial}_{b}^{*}$ are of order zero and independent of $\eta$. Inserting \eqref{5.12} into \eqref{5.10} gives
\begin{multline}\label{5.12a}
 \big\|\chi_{s}[\overline{\partial}_{b},T_{\eta}^{k}]u\big\|^{2}\lesssim (1-\eta)\big(\|\chi_{s}\overline{\partial}_{b}(T_{\eta}^{k}u)\|^{2}+\|\chi_{s}\overline{\partial}_{b}^{*}(T_{\eta}^{k}u)\|^{2}+\|\hat{\chi}_{s}T_{\eta}^{k}u\|^{2}\big) \\
 +C_{\eta}\big(\|\overline{\partial}_{b}u\|_{k-1}^{2}+\|\overline{\partial}_{b}^{*}u\|_{k-1}^{2}+\|u\|_{k-1}^{2}\big) \;.\;\;\;\;\;\;\;
\end{multline}

\smallskip
Next, we similarly address the commutator term with $\overline{\partial}_{b}^{*}$ on the right hand side of \eqref{5.9}. The argument is actually simpler; because  $\sum_{j}\chi_{s}(T_{\eta}^{k}u_{jS})L_{j}=\sum_{j}(\chi_{s}T_{\eta}^{k}u)_{jS}L_{j} = \tilde{L}^{S}_{\chi_{s}T_{\eta}^{k}u}$. Formula \eqref{commutator-formula-dbarbstar} and Proposition \ref{main-result} give
\begin{multline}\label{5.13}
 \|\chi_{s}[\overline{\partial}_{b}^{*},T_{\eta}^{k}]u\|^{2} \\
 \lesssim \sideset{}{'}\sum_{S}\big\|\sum_{j}\chi_{s}\alpha_{\eta}(L_{j})T_{\eta}^{k}u_{jS}\big\|^{2}+C_{\eta}\big(\|\overline{\partial}_{b}u\|_{k-1}^{2}+\|\overline{\partial}_{b}^{*}u\|_{k-1}^{2}+\|u\|_{k-1}^{2}\big) \\
 = \sideset{}{'}\sum_{S}\|\alpha_{\eta}(\tilde{L}^{S}_{\chi_{s}T_{\eta}^{k}u})\|^{2}+C_{\eta}\big(\|\overline{\partial}_{b}u\|_{k-1}^{2}+\|\overline{\partial}_{b}^{*}u\|_{k-1}^{2}+\|u\|_{k-1}^{2}\big) \\
 \lesssim (1-\eta)\big(\|\overline{\partial}_{b}(\chi_{s}T_{\eta}^{k}u)\|^{2}+\|\overline{\partial}_{b}^{*}(\chi_{s}T_{\eta}^{k}u)\|^{2}\big) \;\;\;\;\;\;\;\;\;\;\;\;\;\\
 \;\;\;\;\;\;\;\;\;\;\;\;\;\;\;\;\;\;\;\;\;\;\;\;\;\;\;\;\;\;\;\;\;\;\;\;\;\;\;+C_{\eta}\big(\|\overline{\partial}_{b}u\|_{k-1}^{2}+\|\overline{\partial}_{b}^{*}u\|_{k-1}^{2}+\|u\|_{k-1}^{2}\big) \\
 \;\;\;\;\;\;\;\;\;\lesssim (1-\eta)\big(\|\chi_{s}\overline{\partial}_{b}(T_{\eta}^{k}u)\|^{2}+\|\chi_{s}\overline{\partial}_{b}^{*}(T_{\eta}^{k}u)\|^{2}+\|\hat{\chi}_{s}T_{\eta}^{k}u\|^{2}\big) \\
 + C_{\eta}\big(\|\overline{\partial}_{b}u\|_{k-1}^{2}+\|\overline{\partial}_{b}^{*}u\|_{k-1}^{2}+\|u\|_{k-1}^{2}\big) \;.
\end{multline}
The last inequality is as in \eqref{5.12}. Now we require $\eta$ to be close enough to one so that when we commute $\overline{\partial}_{b}$ and $\overline{\partial}_{b}^{*}$ with $T_{\eta}^{k}$ on the right hand sides of \eqref{5.12a} and \eqref{5.13}, and add the two estimates, $\big\|\chi_{s}[\overline{\partial}_{b},T_{\eta}^{k}]u\big\|^{2}+\|\chi_{s}[\overline{\partial}_{b}^{*},T_{\eta}^{k}]u\|^{2}$ can be absorbed. The conclusion is 
\begin{multline}\label{5.14}
 \big\|\chi_{s}[\overline{\partial}_{b},T_{\eta}^{k}]u\big\|^{2}+\|\chi_{s}[\overline{\partial}_{b}^{*},T_{\eta}^{k}]u\|^{2} \\
 \lesssim (1-\eta)\big(\|\hat{\chi}_{s}T_{\eta}^{k}u\|^{2} 
 + \|\chi_{s}T_{\eta}^{k}\overline{\partial}_{b}u\|^{2}+\|\chi_{s}T_{\eta}^{k}\overline{\partial}_{b}^{*}u\|^{2}\big) \\
\;\;\;\;\;\;\;\;\;\;\;\;\;\;\;\;\;\;\;\;\;\;\;\;\;\;\;\;\;\;\;\; + C_{\eta}\big(\|\overline{\partial}_{b}u\|_{k-1}^{2}+\|\overline{\partial}_{b}^{*}u\|_{k-1}^{2}+\|u\|_{k-1}^{2}\big)  \;.
\end{multline}
Now we take \eqref{5.14} and insert it into the right hand side of \eqref{5.9}. Combining the resulting estimate with \eqref{5.8} and \eqref{5.4} gives
\begin{multline}\label{5.15}
 \|T_{\eta}^{k}u\|^{2} \lesssim (1-\eta)\sum_{s}\|\hat{\chi}_{s}T_{\eta}^{k}u\|^{2}+ \sum_{s}\left(\|\chi_{s}T_{\eta}^{k}\overline{\partial}_{b}u\|^{2}+\|\chi_{s}T_{\eta}^{k}\overline{\partial}_{b}^{*}u\|^{2}\right) \\
 +C_{\eta}\big(\|\overline{\partial}_{b}u\|_{k-1}^{2}+\|\overline{\partial}_{b}^{*}u\|_{k-1}^{2}+\|u\|_{k-1}^{2}\big) \;,
\end{multline}
where on the left hand side $T_{\eta}$ acts on the Euclidean components of $u$ (as in \eqref{5.3} and on the left hand side of \eqref{5.4}), while on the right hand side it acts on the components of $u$ in the special boundary chart associated with $s$. Because the difference is of order $(k-1)$, switching $T_{\eta}$ back to acting on the Euclidean components makes an error that is dominated by the second line in \eqref{5.15}. That is, \eqref{5.15} holds with $T_{\eta}$ acting in Euclidean coordinates. Then the first sum on the right hand side is dominated by $M\|T_{\eta}^{k}u\|^{2}$. Again, for $\eta$ close enough to one, it can be absorbed.The result is 
\begin{equation}\label{5.16}
\|T_{\eta}^{k}u\|^{2}\lesssim \|T_{\eta}^{k}\overline{\partial}_{b}u\|^{2} + \|T_{\eta}^{k}\overline{\partial}_{b}^{*}u\|^{2} 
+C_{\eta}\big(\|\overline{\partial}_{b}u\|_{k-1}^{2}+\|\overline{\partial}_{b}^{*}u\|_{k-1}^{2}+\|u\|_{k-1}^{2}\big) \;.
\end{equation}
We conclude form \eqref{5.16} and \eqref{5.3}:
\begin{equation}\label{5.15a}
 \|u\|_{k}^{2}
 \leq C_{\eta}\left(\|\overline{\partial}_{b}u\|_{k}^{2}+\|\overline{\partial}_{b}^{*}u\|_{k}^{2} + \|u\|_{k-1}^{2}\right)\;.
\end{equation}
This estimate implies the desired conclusion: there is $\eta_{1}$ so that for $\eta>\eta_{1}$, there are constants $C_{\eta}$ such that
\begin{equation}\label{5.17}
 \|u\|_{k}^{2}\leq C_{\eta}\big(\|\overline{\partial}_{b}u\|_{k}^{2} + \|\overline{\partial}_{b}^{*}u\|_{k}^{2}\big)\;,\; u\in C^{\infty}_{(0,q)}(b\Omega)
\end{equation}
(the term $\|u\|_{k-1}^{2}$ is handled as in \eqref{5.2}). Moreover, there are $\delta_{\eta}>0$ such that \eqref{5.17} holds uniformly on the level surfaces $\{\rho_{\eta}=-\delta\}$ for $0\leq \delta<\delta_{\eta}$. With some care, one can check that the above arguments make sense if $u$ is only assumed in $W^{k}_{(0,q)}(b\Omega)$; alternatively, Friedrich's Lemma (\cite{Treves75}, Lemma 25.4) give this conclusion as well. If $\delta>0$, we can say more: as pointed out at the beginning of this section, the estimates in Theorem \ref{main-result} hold, and we may conclude that if $\overline{\partial}_{b}u$ and $\overline{\partial}_{b}^{*}u$ are in $W^{k}(b\Omega_{\eta,\delta})$, then $u$, initially only assumed in $L^{2}_{(0,q)}(b\Omega_{\eta,\delta})$, is indeed in $W^{k}_{(0,q)}(b\Omega_{\eta,\delta})$. Consequently, when $0<\delta<\delta_{\eta}$, \eqref{5.17} is a genuine estimate:
\begin{equation}\label{5.18}
 \|u\|_{k}^{2}\leq C_{\eta}\big(\|\overline{\partial}_{b}u\|_{k}^{2} + \|\overline{\partial}_{b}^{*}u\|_{k}^{2}\big) \;,
\end{equation}
uniformly on $b\Omega_{\eta,\delta}$ for $0<\delta<\delta_{\eta}$. 

\smallskip

Estimate \eqref{5.17} is for $\min\{ q_0,n-1-q_0 \} \le q \le \max\{ q_0,n-1-q_0 \}$. We now address the case $q=\min\{ q_0,n-1-q_0 \}-1$, using an idea from \cite{BoasStraube91b}. A separate argument is needed because we may not have the eigenvalues condition at this level (or this level may be zero). We keep the convention on constants changing from one occurrence to the next. Let $u\in C^{\infty}_{(0,q)}(b\Omega_{\eta,\delta})$, where $\eta$ and $\delta>0$ are from \eqref{5.18}, and $u\perp \ker(\overline{\partial}_{b})$. Then $u=\overline{\partial}_{b}^{*}v$ for $v\in C^{\infty}_{(0,q+1)}(b\Omega_{\eta,\delta})\cap \ker(\overline{\partial}_{b})$ (Theorem \ref{main-result} holds on $b\Omega_{\eta,\delta}$). For $\eta>\eta_{1}$, the estimates from the previous part hold for $(0,q+1)$--forms. In particular, in view of \eqref{5.18},
\begin{equation}\label{5.19}
 \|v\|_{k}^{2} \leq C_{\eta}\|\overline{\partial}_{b}^{*}v\|_{k}^{2} = C_{\eta} \|u\|_{k}^{2}\;.
\end{equation}
The estimates \eqref{5.1} and \eqref{5.3} still hold, and we again only have to estimate $\|T_{\eta}^{k}u\|^{2}$. We have
\begin{equation}\label{5.18a}
\|T_{\eta}^{k}u\|^{2} = \sum_{s}\big(\chi_{s}T_{\eta}^{k}u,T_{\eta}^{k}u\big) 
=\sum_{s}\big(T_{\eta}^{k}(\chi_{s}u),T_{\eta}^{k}u\big)+\mathcal{O}_{\eta}(\|u\|_{k-1}\|u\|_{k})\;.
\end{equation}
Because we will have to use estimate \eqref{5.14}, which is for $T_{\eta}$ acting in special boundary charts, we switch to that setup. If on the right hand side of \eqref{5.18a} we let $T_{\eta}$ act in special boundary charts, we make  an error that is $\mathcal{O}_{\eta}(\|u\|_{k-1}\|u\|_{k})$, and so is benign. Temporarily fixing $s$, and letting $T_{\eta}$ act in the chart associated with $\chi_{s}$, we write
\begin{equation}\label{5.18b}
 \big(T_{\eta}^{k}(\chi_{s}u),T_{\eta}^{k}u\big)= \big(\chi_{s}u,T_{\eta}^{2k}u\big)+\mathcal{O}_{\eta}(\|u\|_{k-1}\|u\|_{k}) \;.
 \end{equation}
 The error terms (or divergence terms) in the $k$--fold integrations by parts are of the form $\big(g_{m}T_{\eta}^{k-m}(\chi_{s}u),T_{\eta}^{m-1}T_{\eta}^{k}u\big)$, for some smooth function $g_{m}$. In view of the duality between $(W^{m-1}_{0})_{(0,q)}(U_{s}\cap b\Omega_{\eta,\delta})$ and $W^{-m+1}_{(0,q)}(U_{s}\cap b\Omega_{\eta,\delta})$ (via coefficientwise pairing), these terms are dominated by $\|g_{m}T_{\eta}^{k-m}(\chi_{s}u)\|_{W^{m-1}_{(0,q)}(U_{s}\cap b\Omega_{\eta,\delta})}\|T_{\eta}^{m-1}T_{\eta}^{k}u\|_{W^{-m+1}_{(0,q)}(U_{s}\cap b\Omega_{\eta,\delta})} \leq C_{\eta}\|u\|_{k-1}\|T_{\eta}^{k}u\|$ $\leq C_{\eta}\|u\|_{k-1}\|u\|_{k}$. The norms in the first $C_{\eta}$--term are over $U_{s}\cap b\Omega_{\eta,\delta}$; in the second, the norms are global. The main term on the right hand side of \eqref{5.18b} equals
\begin{multline}\label{5.18c}
 \big(\chi_{s}u,T_{\eta}^{2k}\overline{\partial}_{b}^{*}v\big)=
\big(\chi_{s}u,\overline{\partial}_{b}^{*}T_{\eta}^{2k}v\big)-\big(\chi_{s}u,[\overline{\partial}_{b}^{*},T_{\eta}^{2k}]v\big) \\
=\big(\overline{\partial}_{b}(\chi_{s}u),T_{\eta}^{2k}v\big) -\big(\chi_{s}u,[\overline{\partial}_{b}^{*},T_{\eta}^{2k}]v\big) \\
=\big(\overline{\partial}_{b}\chi_{s}\wedge u, T_{\eta}^{2k}v\big)+\big(\chi_{s}\overline{\partial}_{b}u,T_{\eta}^{2k}v\big)-\big(\chi_{s}u,[\overline{\partial}_{b}^{*},T_{\eta}^{2k}]v\big) \;.
 \end{multline}
Starting with \eqref{5.18a}, we obtain
 \begin{multline}\label{5.18d}
  \|T_{\eta}^{k}u\|^{2}\lesssim \left|\sum_{s}\left(\big(\overline{\partial}_{b}\chi_{s}\wedge u, T_{\eta}^{2k}v\big)+\big(\chi_{s}\overline{\partial}_{b}u,T_{\eta}^{2k}v\big)-\big(\chi_{s}u,[\overline{\partial}_{b}^{*},T_{\eta}^{2k}]v\big)\right)\right| \\
  +C_{\eta}\|u\|_{k-1}\|u\|_{k}
 \end{multline}
Because the difference between between $T_{\eta}^{2k}$ acting in a special boundary chart to the globally defined term where $T_{\eta}^{2k}$ acts in Euclidean coordinates is of order $(2k-1)$, and $\sum_{s}\overline{\partial}\chi_{s}=0$, the contribution to the right hand side of \eqref{5.18a} that originates with the first term on the right hand side of \eqref{5.18d} is dominated by $C_{\eta}\|\overline{\partial}_{b}\chi_{s}\wedge u\|_{k-1}\|D^{2k-1}(\hat{\chi}_{s}v)\|_{-k+1}\leq C_{\eta}\|u\|_{k-1}\|v\|_{k}$ ($D^{2k-1}$ is some differential operator of order $(2k-1)$; we have again used duality as above). The contribution from the middle term on the right hand side of the first line in \eqref{5.18d} is estimated by (again via a Sobolev space pairing on $U_{s}\cap b\Omega_{\eta,\delta}$) 
\begin{equation}\label{5.18e}
\left|\sum_{s}\big(\chi_{s}\overline{\partial}_{b}u,T_{\eta}^{2k}v\big)\right|\lesssim
\sum_{s}\|\chi_{s}\overline{\partial}_{b}u\|_{k}\|T_{\eta}^{2k}v\|_{-k} \leq C_{\eta}\|\overline{\partial}_{b}u\|_{k}\|v\|_{k} \;;
\end{equation}
again, the norms are first on $U_{s}\cap b\Omega_{\eta,\delta}$, and then on all of $b\Omega_{\eta,\delta}$. For the last term on the right hand side of \eqref{5.18d}, note that
\begin{equation}\label{5.22}
[\overline{\partial}_{b}^{*},T_{\eta}^{2k}]=[\overline{\partial}_{b}^{*},T_{\eta}^{k}]T_{\eta}^{k}+T_{\eta}^{k}[\overline{\partial}_{b}^{*},T_{\eta}^{k}] = 2T_{\eta}^{k}[\overline{\partial}_{b}^{*},T_{\eta}^{k}]-\big[[\overline{\partial}_{b}^{*},T_{\eta}^{k}],T_{\eta}^{k}\big] \;.
\end{equation}
The last term here is of order $(2k-1)$. Therefore,
\begin{multline}\label{5.23}
  \big|\big(\chi_{s}u,[\overline{\partial}_{b}^{*},T_{\eta}^{2k}]v\big)\big| \lesssim \big|\big(\chi_{s}u,T_{\eta}^{k}[\overline{\partial}_{b}^{*},T_{\eta}^{k}]v\big)\big| + C_{\eta}\|u\|_{k}\|v\|_{k-1} \\
   \lesssim \big|\big(T_{\eta}^{k}\chi_{s}u,[\overline{\partial}_{b}^{*},T_{\eta}^{k}]v\big)\big|+C_{\eta}\big(\|u\|_{k-1}\|[\overline{\partial}_{b}^{*},T_{\eta}^{k}]v\|+\|u\|_{k}\|v\|_{k-1}\big) \\
   \lesssim \big|\big(T_{\eta}^{k}u,\chi_{s}[\overline{\partial}_{b}^{*},T_{\eta}^{k}]v\big)\big|+C_{\eta}\big(\|u\|_{k-1}\|v\|_{k}+\|u\|_{k}\|v\|_{k-1}\big) \;.
\end{multline}
The integration by parts is handled as in \eqref{5.18b}. To estimate the main  term on the right hand side, we use \eqref{5.14} (with $u$ replaced by $v$):  
\begin{equation}\label{5.24}
\big\|\chi_{s}[\overline{\partial}_{b}^{*},T_{\eta}^{k}]v\big\|^{2}  
\lesssim (1-\eta)\big(\|\chi_{s}T_{\eta}^{k}u\|^{2}+ \|\hat{\chi}_{s}T_{\eta}^{k}v\|^{2}\big)
+C_{\eta}\big(\|u\|_{k-1}^{2}+\|v\|_{k-1}^{2}\big) \;.
\end{equation}
We have used that $\overline{\partial}_{b}v=0$ and $\overline{\partial}_{b}^{*}v=u$. Therefore, the main term on the right hand side of \eqref{5.23} can be estimated as follows:
\begin{multline}\label{5.24a}
 \big|\big(T_{\eta}^{k}u,\chi_{s}[\overline{\partial}_{b}^{*},T_{\eta}^{k}]v\big)\big| \leq (1-\eta)^{1/2}\|\hat{\chi}_{s}T_{\eta}^{k}u\|^{2}+(1-\eta)^{-1/2}\big\|\chi_{s}[\overline{\partial}_{b}^{*},T_{\eta}^{k}]v\big\|^{2} \\
 \lesssim (1-\eta)^{1/2}\big(\|\hat{\chi}_{s}T_{\eta}^{k}u\|^{2}+ \|\hat{\chi}_{s}T_{\eta}^{k}v\|^{2}\big)+C_{\eta}\big(\|u\|_{k-1}^{2}+\|v\|_{k-1}^{2}\big)
 \end{multline}
(note that $\hat{\chi}_{s}$ dominates $\chi_{s}$). Starting with \eqref{5.18d} and summing the estimates for the various terms on the right hand side results in 
\begin{multline}\label{5.24b}
\|T_{\eta}^{k}u\|^{2}\lesssim (1-\eta)^{1/2}\sum_{s}\big(\|\hat{\chi}_{s}T_{\eta}^{k}u\|^{2}+\|\hat{\chi}_{s}T_{\eta}^{k}v\|^{2}\big)  \\
+C_{\eta}\big(\|\overline{\partial}_{b}u\|_{k}\|u\|_{k}+\|u\|_{k-1}\|u\|_{k}\big) \;;
\end{multline}
we have used \eqref{5.19} for $\|v\|_{k}^{2}$ and $\|v\|_{k-1}^{2}$, and $\|u\|_{k-1}^{2}\leq \|u\|_{k}\|u\|_{k-1}$. On the left hand side of \eqref{5.24b}, $T_{\eta}$ acts on Euclidean components, whereas on the right hand side of the first line, it acts on the components in the chart associated with $\chi_{s}$.  However, by what is now a familiar observation, when we let it act on Euclidean components also on the right hand side, we make an error that is covered by the last error term in the second line. Accordingly, \eqref{5.24b} also holds with $T_{\eta}$ acting on Euclidean components everywhere. Denote by $\tilde{C}_{\eta}$ the constant from \eqref{5.3} for the current value of $\eta$. The error terms satisfy
\begin{multline}\label{5.24bb}
 C_{\eta}\big(\|\overline{\partial}_{b}u\|_{k}\|u\|_{k}+\|u\|_{k-1}\|u\|_{k}\big)\leq \frac{(1-\eta)^{1/2}}{\tilde{C}_{\eta}}\|u\|_{k}^{2}+C_{\eta}\big(\|\overline{\partial}_{b}u\|_{k}^{2}+\|u\|_{k-1}^{2}\big) \\
 \lesssim \frac{(1-\eta)^{1/2}}{\tilde{C}_{\eta}}\|u\|_{k}^{2}+C_{\eta}\big(\|\overline{\partial}_{b}u\|_{k}^{2}+\|u\|^{2}\big)\lesssim \frac{(1-\eta)^{1/2}}{\tilde{C}_{\eta}}\|u\|_{k}^{2}+C_{\eta}\|\overline{\partial}_{b}u\|_{k}^{2}\;.
 \end{multline}
The reader is reminded that constants (in this case $C_{\eta}$) are allowed to change their values form one occurrence to the next. For the first inequality, we have used the s.c.--l.c. estimate, for the second we have interpolated $\|u\|_{k-1}^{2}\leq((1-\eta)^{1/2}/\tilde{C}_{\eta})\|u\|_{k}^{2} + C_{\eta}\|u\|^{2}$, and for the third the estimate $\|u\|^{2}\lesssim \|\overline{\partial}_{b}u\|^{2}\leq \|\overline{\partial}_{b}u\|_{k}^{2}$ ($u\perp\ker(\overline{\partial}_{b})$ and \eqref{basic}). With \eqref{5.24bb}, \eqref{5.24b} becomes
\begin{equation}\label{5.24c}
\|T_{\eta}^{k}u\|^{2}\lesssim (1-\eta)^{1/2}\big(\|T_{\eta}^{k}u\|^{2}+\|T_{\eta}^{k}v\|^{2}+(1/\tilde{C_{\eta}})\|u\|_{k}^{2}\big)  
+C_{\eta}\|\overline{\partial}_{b}u\|_{k}^{2} \;.
\end{equation}  
From \eqref{5.16}, we have
\begin{equation}\label{5.26}
 \|T_{\eta}^{k}v\|^{2}\lesssim \|T_{\eta}^{k}u\|^{2}+C_{\eta}\|u\|_{k-1}^{2} \;,
\end{equation}
again because $\overline{\partial}_{b}v=0$, $\overline{\partial}_{b}^{*}v=u$, and \eqref{5.19}. Inserting this estimate into \eqref{5.24c}, restricting $\eta$ close enough to one, and absorbing $(1-\eta)^{1/2}\|T_{\eta}^{k}u\|^{2}$ yields
\begin{multline}\label{5.26a}
 \|T_{\eta}^{k}u\|^{2}\lesssim (1-\eta)^{1/2}(1/\tilde{C_{\eta}})\|u\|_{k}^{2}+C_{\eta}\big(\|\overline{\partial}_{b}u\|_{k}^{2}+\|u\|_{k-1}^{2}\big) \\
 \lesssim (1-\eta)^{1/2}(1/\tilde{C_{\eta}})\|u\|_{k}^{2}+C_{\eta}\|\overline{\partial}_{b}u\|_{k}^{2} \;;
\end{multline}
where we have used the same argument as in \eqref{5.24bb} for the $\|u\|_{k-1}^{2}$--term. Inserting \eqref{5.26a} into \eqref{5.3} and absorbing the $\|u\|_{k}^{2}$ term gives the desired estimate
\begin{equation}\label{5.31}
 \|u\|_{k}^{2} \lesssim C_{\eta}\|\overline{\partial}_{b}u\|_{k}^{2} \;;\;u\in C^{\infty}_{(0,q)}(b\Omega_{\eta,\delta})\;,\;u\perp\ker(\overline{\partial}_{b})\;.
\end{equation}
If $u$ is only assumed in $L^{2}_{(0,q)}(b\Omega_{\eta,\delta})$, but $\overline{\partial}_{b}u \in W^{k}_{(0,q+1)}(b\Omega_{\eta,\delta})$ and $u\perp\ker(\overline{\partial}_{b})$, then $u\in W^{k}_{(0,q)}(b\Omega_{\eta,\delta})$ (Theorem 1 holds on $b\Omega_{\eta,\delta}$, we just don't control the constants). Using that the pairings $(f,g)$ are well defined for $f\in W^{t}_{0}$ and $g\in W^{-t}$, it is now not hard to check that the arguments above hold when $u$ is only assumed in $W^{k}_{(0,q)}(b\Omega_{\eta,\delta})$. Therefore,there is $\eta_{2}$,
$\eta_{1}<\eta_{2}$ such that for $\eta_{2}<\eta$, there is $\delta_{\eta}>0$
so that we have the (genuine) estimate
\begin{equation}\label{5.32}
\|u\|_{k}^{2}\lesssim C_{\eta}\|\overline{\partial}_{b}u\|_{k}^{2}\;,\;u\perp\ker(\overline{\partial}_{b})\;,
\end{equation}
uniformly on $b\Omega_{\eta,\delta}$ for $\eta_{2}<\eta$ and $0<\delta<\delta_{\eta}$.

\smallskip

The case $q=\max\{q_{0}, n-1-q_{0}\}+1$ is handled analogously, with the commutators $[\overline{\partial}_{b}^{*},T_{\eta}^{2k}]$ and $[\overline{\partial}_{b}^{*},T_{\eta}^{k}]$ replaced by $[\overline{\partial}_{b},T_{\eta}^{2k}]$ and $[\overline{\partial}_{b},T_{\eta}^{k}]$, respectively. It is also interesting to use that estimates hold at symmetric form levels. We only sketch the argument. Consider the isometries $A_{p,q}$ from  Proposition 1 in \cite{Biard-Straube-2019}, in particular $A_{n,q}: L^{2}_{(n,q)}(b\Omega)\rightarrow L^{2}_{(0,n-1-q)}(b\Omega)$. These operators also preserve the Sobolev spaces (proof of Theorem 1 in \cite{Biard-Straube-2019}). Note that in our case $(n-1-q)=\min\{q_{0},n-1-q\}-1$. The latter is the form level where we have shown estimates. Thus we have 
\begin{multline}\label{5.33}
 \|u\|_{k}^{2}=\|u\wedge dz_{1}\wedge \cdots\wedge dz_{n}\|_{k}^{2}=\|A_{n,q}(u\wedge dz_{1}\wedge \cdots\wedge dz_{n})\|_{k}^{2} \\
 \lesssim \|\overline{\partial}_{b}A_{n,q}(u\wedge dz_{1}\wedge \cdots\wedge dz_{n}\|_{k}^{2}= \|A_{n,q-1}\overline{\partial}_{b}^{*}(u\wedge dz_{1}\wedge \cdots\wedge dz_{n})\|_{k}^{2}\\
 = \|\overline{\partial}_{b}^{*}(u\wedge dz_{1}\wedge \cdots\wedge dz_{n})\|_{k}^{2}
 =\|(\overline{\partial}_{b}^{*}u)\wedge dz_{1}\wedge \cdots\wedge dz_{n}\|_{k}^{2}=\|\overline{\partial}_{b}^{*}u\|_{k}^{2}\;.
\end{multline}
We have used that $u\rightarrow u\wedge dz_{1}\wedge \cdots\wedge dz_{n}$
takes tangential forms to tangential forms and preserves orthogonality to $\ker(\overline{\partial}_{b}^{*})$, that $A_{n,q}$, as an isometry, preserves orthogonality in general, and that $A_{n,q}$ and $A_{n,q-1}$  intertwine $\overline{\partial}_{b}$ and $\overline{\partial}_{b}^{*}$ (so that $A_{nq}(u\wedge dz_{1}\wedge \cdots\wedge dz_{n})\perp \ker(\overline{\partial}_{b})$). Inspection of the arguments in \cite{Biard-Straube-2019} shows that for $\eta$ fixed, there is $\delta_{\eta}$ such that the above estimates can be done uniformly on $b\Omega_{\eta,\delta}$ for $0<\delta<\delta_{\eta}$.

\smallskip

This completes the first part of our argument. We have shown that there is $\eta_{3}<1$ such that for $\eta$ with $\eta_{3}<\eta<1$, there exists $\delta_{\eta}>0$ such that the estimates (1) in Theorem \ref{main-result} hold uniformly on $b\Omega_{\eta,\delta}$ for $\eta_{3}<\eta<1$ and $0<\delta<\delta_{\eta}$.

\smallskip

We now address the estimates (2) in Theorem \ref{main-result}. These will follow from what we have already shown together with the results of \cite{Straube25}. Specifically, part (2) of Theorem 1.1. there says that estimates (2) in Theorem \ref{main-result} hold if and only if estimates (1) hold. This means for the pairs $(\eta,\delta)$ from the previous paragraph, the estimates (2) in Theorem \ref{main-result} hold on $b\Omega_{\eta,\delta}$. We need to know that in addition, they hold with uniform constants. The argument in \cite{Straube25} relies on the weighted theory due to Nicoara (\cite{Nicoara06}) and Harrington-Raich (\cite{Harrington-Raich-2011}), see also \cite{Harrington-Peloso-Raich-2015}. In particular, formula (6) in \cite{Harrington-Peloso-Raich-2015}, with $(q-1)$ replaced by $q$ (compare also formula (9) and estimate (10) in \cite{Straube25}), shows that for $0\leq q\leq (n-2)$, we have
\begin{equation}\label{5.34}
 \|S_{q}u\|_{k}^{2}\lesssim \|F_{t}S_{q,t}F_{t}^{-1}u\|^{2}+\|[\overline{\partial}_{b},F_{t}]S_{q,t}F_{t}^{-1}\|_{k}^{2}\;,
\end{equation}
where $S_{q,t}$ is the weighted Szeg\"{o} projection and $F_{t}$ and $F_{t}^{-1}$ are pseudodifferential operators of order zero (Corollary 4.6 and its proof in \cite{Nicoara06}). Given $k$, there is $t_{k}$ so that for $t>t_{k}$, $S_{q,t}$ is continuous in $W^{k}_{(0,q)}(b\Omega)$ and \eqref{5.34} implies continuity of $S_{q}$ (the commutator is of order zero). When $q=(n-1)$, the substitute for \eqref{5.34} is estimate (11) in \cite{Straube25}:
\begin{equation}\label{5.35}
 \|S_{n-1}u\|_{k}^{2}\lesssim \|S_{n-1,t}u\|_{k}^{2} + \|(\overline{\partial}_{b}^{*}-\overline{\partial}_{b,t}^{*}))(I-S_{n-1,t})u\|_{k}^{2}\;.
\end{equation}
It is also observed there that $(\overline{\partial}_{b}^{*}-\overline{\partial}_{b,t}^{*})=[\overline{\partial}_{b}^{*},F_{t}]F_{t}^{-1}$, and so is of order zero. Uniformity of the constants over $b\Omega_{\eta,\delta}$ follows from how the weighted theory is built. Roughly speaking, the weights are formed by multiplication with $e^{-t|z|^{2}}$ or $e^{t|z|^{2}}$ on the positive and negative microlocalizations, respectively. The operators $F_{t}$ play the role of multiplication by $e^{-t|z|^{2}}$ in Kohn's weighted theory for $\overline{\partial}$. These constructions can be done with uniform estimates over boundaries of approximating subdomains, decreasing $\delta_{\eta}$ as needed.

\smallskip

 We have now shown that there is $\eta_{3}<1$ such that for $\eta$ with $\eta_{3}<\eta<1$, there exists $\delta_{\eta}>0$ such that $(1)$ and $(2)$ in Theorem \ref{main-result} hold uniformly on $b\Omega_{\eta,\delta}$ for $\eta_{3}<\eta<1$ and $0<\delta<\delta_{\eta}$. To conclude the proof of Theorem \ref{main-result}, pick an $\eta$ with $\eta_{3}<\eta<1$. The uniform estimates on $b\Omega_{\eta,\delta}$ for $0<\delta<\delta_{\eta}$ now transfer to $b\Omega$ as in section 4 of \cite{BoasStraube91b}. This concludes the proof of Theorem \ref{main-result}.

\section{Appendix: Kohn--Morrey--H\"{o}rmander on the boundary}\label{appendix}

We will need the following Kohn--Morrey--H\"{o}rmander type estimate for the boundary.
\begin{proposition}\label{proposition-kohn-morrey-hormander}
    Let $\O\subset\C^n$ be a smooth bounded pseudoconvex domain; let $u = \sum_{|J|=q}' u_J\overline{\omega}_J$ $\in \dom(\overline{\partial}_{b})\cap\dom(\overline{\partial}_{b}^{*})\subset L^2_{(0,q)}(b\Omega)$ be supported in a special boundary chart; and let $\varphi\in C^\infty(b\O)$. Then 
    \begin{multline}\label{tkmh}
        \normw{\dbarb u}^2 + \normw{\dbarbstarweighted u}^2  \gtrsim \;\normw{\overline{L}u}^2 + 2\text{Re} \isum_{K}\sum_{j,k}\int_{b\O}  c_{jk}\left( T - \frac{1}{2}\frac{\partial\varphi}{\partial\nu} \right) u_{jK}\overline{u_{kK}}\weight \\
        +  2\isum_{K}\int_{b\O} (\partial\overline{\partial}\varphi)\big(\tilde{L}_{u}^{K}\wedge\overline{\tilde{L}_{u}^{K}}\big)\weight  +  \mathcal{O}(\normw{u}^2),
\end{multline}
where $\normw{\cdot}$ is the weighted norm with respect to the weight function $\weight$, $\normw{\overline{L}u}^2$ is the sum $\sum_{j}\sum_{J}' \normw{\overline{L}_ju_J}^2$, $K$ is a multi-index of length $q-1$, $T = L_n-\overline{L}_n$, $c_{jk}$ are the coefficients of the Levi form in the basis $\{L_1,\dots,L_{n-1}\}$, $\tilde{L}_{u}^{K}$ are the vector fields from section \ref{ICDA}, and the constant appearing in the $\mathcal{O}$ term is independent of $\varphi$.
\end{proposition}
Estimate \eqref{tkmh} is a special case of an estimate in \cite{Harrington-Raich-2011}. We thank Andy Raich for pointing us to the second formula from the bottom on page 150. Lets call it (HR). In (HR), take $\lambda^{+}=\varphi$, $m=0$ (since $\Omega$ is pseudoconvex), and $\phi=u$. Then the left hand side of (HR) gives the left hand side of \eqref{tkmh}, the first sum in the fist line of (HR) is absent, and the second sum equals the first term on the right hand side of \eqref{tkmh}. The coefficients $h^{+}_{jk}$ are defined in Lemma 4 in \cite{Harrington-Raich-2011} as $h^{+}_{jk}=c_{jk}-\delta_{jk}\frac{1}{q}\sum_{l=1}^{m}c_{ll}=c_{jk}$ (because $m=0$), where $c_{jk}$ are the coefficients of the Levi form; i.e. they are the same as in \eqref{tkmh}. Next, $\Gamma^{\lambda^{+}}_{jk}=\Theta^{\lambda^{+}}_{jk}$ by the formula immediately preceding (HR) (since $m=0$). Definition (1) of $\Theta^{\lambda^{+}}$ and Proposition 3.1 now show that $\sum_{I\in\mathcal{J}_{q-1}}\sum_{j,k=1}^{n-1}\text{Re}(\Gamma^{\lambda^{+}}_{jk}\phi_{jI},\phi_{kI})=\sum^{'}_{K}(\partial\overline{\partial}\varphi)\big(\tilde{L}_{u}^{K}\wedge\overline{\tilde{L}_{u}^{K}}\big)$. Therefore, (HR) translates to
\begin{multline}\label{6.2}
\normw{\dbarb u}^2 + \normw{\dbarbstarweighted u}^2 = \normw{\overline{L}u}^2 + \text{Re} \isum_{K}\sum_{j,k}\int_{b\O}  c_{jk}\left( T - \frac{1}{2}\frac{\partial\varphi}{\partial\nu} \right) u_{jK}\overline{u_{kK}}\weight  \\
+ \isum_{K}\int_{b\O}(\partial\overline{\partial}\varphi)\big(\tilde{L}_{u}^{K}\wedge\overline{\tilde{L}_{u}^{K}}\big)\weight + E(u)\;,
\end{multline}
where the error term satisfies the estimate $|E(u)|\leq C\big(\|u\|_{\varphi}^{2} +|(\overline{L}u,u)_{\varphi}|\big)$, where the constant $C$ is independent of $\varphi$ and $u$ (see page 147). Estimating $|(\overline{L}u,u)_{\varphi}|$ as usual by $\frac{1}{2}\|\overline{L}u\|_{\varphi}^{2}+\mathcal{O}(\|u\|_{\varphi}^{2})$ gives
\begin{multline}\label{6.3}
 \normw{\dbarb u}^2 + \normw{\dbarbstarweighted u}^2  \gtrsim \frac{1}{2}\normw{\overline{L}u}^2 + \text{Re} \isum_{K}\sum_{j,k}\int_{b\O}  c_{jk}\left( T - \frac{1}{2}\frac{\partial\varphi}{\partial\nu} \right) u_{jK}\overline{u_{kK}}\weight  \\
+ \isum_{K}\int_{b\O}(\partial\overline{\partial}\varphi)\big(\tilde{L}_{u}^{K}\wedge\overline{\tilde{L}_{u}^{K}}\big)\weight +  \mathcal{O}(\normw{u}^2)\;;
\end{multline}
upon multiplication by $2$, this is \eqref{tkmh}. The various integrations by parts and estimates can be done uniformly over approximating domains.

\vskip .5cm

\providecommand{\bysame}{\leavevmode\hbox to3em{\hrulefill}\thinspace}

\end{document}